\documentclass[reqno,12pt]{amsart}

\usepackage{amsmath}
\usepackage{aliascnt}
\usepackage{amsfonts}
\usepackage{bm,bbm}
\usepackage{amsthm}
\usepackage[dvipsnames, svgnames]{xcolor}
\usepackage{dsfont}
\usepackage{graphicx}
\usepackage{float}
\usepackage[export]{adjustbox}
\usepackage{fancyref}
\usepackage{algorithm2e}
\usepackage[hidelinks]{hyperref}
\usepackage[colorinlistoftodos,prependcaption,textsize=small]{todonotes}
\usepackage[title,titletoc,toc]{appendix}
\usepackage{mathrsfs}
\usepackage{microtype}
\usepackage{enumerate} 
\usepackage{listings}
\usepackage{bbm}
\usepackage{tikz-cd} 
\usepackage{mathabx} 
\usepackage[a4paper,bindingoffset=0.5cm,left=1.9cm,right=1.9cm,top=2.5cm,bottom=2cm,footskip=.8cm]{geometry}
\usepackage{pgfplots}

\usepackage{tikzexternal}
\definecolor{color11}{HTML}{aedea7}
\definecolor{color12}{HTML}{74c476}
\definecolor{color13}{HTML}{37a055}
\definecolor{color14}{HTML}{c6c6c6}
\definecolor{color15}{HTML}{969696}
\definecolor{color16}{HTML}{686868}

\definecolor{color21}{HTML}{98d594}
\definecolor{color22}{HTML}{4bb062}
\definecolor{color23}{HTML}{157f3b}
\definecolor{color24}{HTML}{00441b}

\definecolor{color31}{HTML}{bce4b5}
\definecolor{color32}{HTML}{8ed08b}
\definecolor{color33}{HTML}{57b668}
\definecolor{color34}{HTML}{2c944c}
\definecolor{color35}{HTML}{057130}
\definecolor{color36}{HTML}{00441b}

\allowdisplaybreaks

\newtheorem{theorem}{Theorem}[section]

\newaliascnt{lemma}{theorem}

\aliascntresetthe{lemma}

\newaliascnt{corollary}{theorem}

\aliascntresetthe{corollary}

\newaliascnt{proposition}{theorem}
\newtheorem{proposition}[proposition]{Proposition}
\aliascntresetthe{proposition}

\newaliascnt{remark}{theorem}
\newtheorem{remark}[remark]{Remark}
\aliascntresetthe{remark}

\newaliascnt{assumption}{theorem}
\newtheorem{assumption}[assumption]{Assumption}
\aliascntresetthe{assumption}

\newaliascnt{definition}{theorem}

\aliascntresetthe{definition}

\newaliascnt{example}{theorem}
\newtheorem{example}[example]{Example}
\aliascntresetthe{example}

\newcommand{\diag}{\textrm{diag}}

\newcommand{\rhoh}{\hat{\rho}}
\newcommand{\rhob}{\bar{\rho}}

\newcommand{\R}{\mathbbm{R}}

\newcommand{\PP}{\mathbbm{P}}
\newcommand{\E}{\mathbbm{E}}

\newcommand{\Z}{\mathbbm{Z}}

\newcommand*\diff{\mathop{}\!\mathrm{d}}

\newcommand{\delims}[4]{\mathopen#1#2#4\mathclose#1#3}
\newcommand{\norm}[2][]{\delims{#1}\lVert\rVert{#2}}	

\newcommand{\red}[1]{\textcolor{red}{#1}}

\let\lim\relax

\DeclareMathOperator*{\lim}{\vphantom{p}lim}

\definecolor{light-gray}{gray}{0.65}
\definecolor{darkgreen}{HTML}{44CC44}

\author{Michel Mandjes and Jaap Storm}
\title[A diffusion-based analysis of   a multi-class road traffic network]{A diffusion-based analysis of  \\ a multi-class road traffic network}
\date{\today}

\begin{document}

\begin{abstract}
This paper studies a stochastic model that describes the evolution of vehicle densities in a road network. It is consistent with the class of (deterministic) kinematic wave models, which describe traffic flows on the basis of conservation laws that incorporate the macroscopic fundamental diagram (a functional relationship between vehicle density and flow). Our setup is capable of handling multiple types of vehicle densities, with general macroscopic fundamental diagrams, on a network with arbitrary topology. 

\noindent Interpreting our system as a spatial population process, we derive, under a natural scaling, fluid and diffusion limits. More specifically, the vehicle density process can be approximated with a suitable Gaussian process, which yield accurate normal approximations to the joint (in the spatial and temporal sense) vehicle density process. The corresponding means and variances can be computed efficiently. Along the same lines, we develop an approximation to the vehicles' travel-time distribution between any given origin and destination pair. Finally, we present a series  of numerical experiments that demonstrate the accuracy of the approximations and  illustrate the usefulness of the results.


\vspace{3mm}


\vspace{3mm}

\noindent
{\sc Affiliations.} 
Michel Mandjes ({\tt \footnotesize m.r.h.mandjes@uva.nl}) is with Korteweg-de Vries Institute for Mathematics, University of Amsterdam, Science Park 904, 1098 XH Amsterdam, the Netherlands. He is also with E{\sc urandom}, Eindhoven University of Technology, Eindhoven, the Netherlands, and Amsterdam Business School, Faculty of Economics and Business, University of Amsterdam, Amsterdam, the Netherlands. Jaap Storm ({\tt \footnotesize p.j.storm@vu.nl}) is with {the} Department of Mathematics, Vrije Universiteit Amsterdam,
De Boelelaan {1111}, 1081 HV Amsterdam, the Netherlands. Their research is partly funded by NWO Gravitation project N{\sc etworks}, grant number 024.002.003.

\end{abstract}

\maketitle

\section{Introduction}
A substantial body of literature focuses on describing and predicting the dynamics of vehicles on road traffic networks.
A broad range of traffic flow models has been developed, each of them focusing on specific aspects.  Their ultimate goal lies in the development of mechanisms that effectively control streams of vehicles.

The majority of the existing traffic flow models is of a deterministic nature. These range from macroscopic models (typically in the form of partial differential equations that describe vehicles as continuous flows that obey physical laws) to microscopic models (incorporating the driving behavior pertaining to individual vehicles). However, as was pointed out in, e.g.,~\cite{QU2017}, besides physical laws, traffic flows are also strongly affected by various microscopic  variables, such as the different perceptions, moods, responses, and driving habits of individual car drivers. This realization has led to the consensus \cite[Section~1]{QU2017} that such microscopic variables should be modeled as random variables. Thus, to accurately describe streams of vehicles in a road traffic network, {\it stochastic} traffic flow models are needed. In addition, as argued in great detail in \cite{JL2012}, probabilistic traffic flow models are particularly useful from the viewpoint of traffic simulation, estimation, and online control. 

In relation to the choice of a suitable probabilistic model, as was pointed out in \cite{JL2012}, a frequently used type of stochastic traffic flow models  are macroscopic deterministic models with added noise. These, however, result in an inconsistency with the corresponding deterministic models in the case of non-linear dynamics, and possibly lead to the undesirable feature of negative sample paths. Truncation of the noise, to prevent negative sample paths from happening, would only exacerbate the above-mentioned inconsistency. Turning to microscopic stochastic traffic flow models allows for explicit modeling of stochastic behavior.  Examples are cellular automata models  \cite{maerivoet2005traffic}, with a seminal example being the celebrated Nagel-Schreckenberg model, and car following models \cite[Section 3]{WageningenKessels2015genealogy}. However, due to the complexity of the underlying dynamics, for such models the computation of the vehicle density distribution at a given time is typically intractable for instances of a realistic size.

An approach frequently relied upon in macroscopic models, based on conservation laws, makes use of the so-called {\it fundamental diagram}, which describes the functional relation between the vehicle density and velocity. The origin of such {\it kinematic wave models}, which succeed in replicating empirical traffic phenomena, lies in the seminal papers by Lighthill and Whitham \cite{lw1955} and Richards \cite{rich1956}. In the kinematic wave framework, the flow of traffic mass is modeled by a conservation law, with the additional feature that the velocity of mass at a position in space is a function of the density of this mass. Obviously, to facilitate accurate performance predictions, having a handle on the precise shape of the fundamental diagram is crucial. In particular, researchers try to reproduce physical phenomena like the \textit{scatter} observed in empirical studies, the \textit{drop in capacity} at the onset of congestion, \textit{hysteresis} while accelerating and decelerating, and the effect of multiple types of vehicles; see{, e.g.,~}\cite{cb2003, WageningenKessels2015genealogy}. 

While traditionally the focus was on single-class models \cite{daganzo1994cell,drake,smul}, in which all vehicles are essentially exchangeable, in {practice, }multiple classes should be distinguished. This explains why more recently deterministic multi-class models \cite{cb2003, li2008, ww2002}  have been developed, which capture the heterogeneity of vehicles and drivers.
However, these models come with new complications, in that they do not automatically satisfy the \textit{anisotropic} (informally meaning that information travels slower than the fastest class of vehicles) and even \textit{hyperbolic} (informally meaning that any single vehicle has only local influence on the system dynamics) properties of conservation laws. As discussed in{, e.g.,~}\cite{DA95, WageningenKessels2015genealogy}, realistic traffic flow models should obey these properties.

\vspace{3mm}

The main conclusion of the above, is that there is a clear demand for stochastic multi-class traffic flow models that are flexible enough to cover various generic physical phenomena. The discrete-space stochastic model proposed in \cite{JL2012,JL2013}, is referred to by \cite{QU2017} as one of the few models that provides an explicit expression for the vehicle density distribution. In addition, it is consistent with kinematic wave models, and it aligns with the above-mentioned physical effects that are inherent to the fundamental diagram. {While this model has these attractive properties, several significant improvements can be made. In the first place, the focus is exclusively on the class of fundamental diagrams that involve just single-type vehicle densities, not covering the practically relevant multi-type case. In the second place, the model considers isolated road segments rather than more general networks. The third issue relates to the underlying analysis that is based on a diffusion approach, leading to a Gaussian approximation of vehicle densities. While the claimed diffusion, under the proposed scaling, is correct, there is a need for addressing the formal details of the mathematical underpinnings. Finally, in \cite{JL2012,JL2013} the focus is fully on describing the distribution of the vehicle densities, whereas no attention is paid to computing the distribution of travel times between given origin-destination pairs.}

Our work concerns a rigorous analysis of a generalized version of the cellular stochastic traffic flow model that was introduced and analyzed in \cite{JL2012,JL2013}, remedying the issues identified in the previous paragraph. More specifically, the contributions of our paper are the following:
\begin{itemize}
\item[$\circ$]
Unlike in the single-class model of \cite{JL2012,JL2013}, we focus on a setup capable of handling a {\it multi-class} macroscopic fundamental diagram. While in our presentation we focus on a road segment consisting of multiple cells, we point out how the analysis naturally extends to  networks of arbitrary size and structure. 
\item[$\circ$]
Moreover, we embed the model in a rigorous mathematical framework that allows us to appeal to the theory of {\it spatial population processes} \cite{kurtz1981approximation}. We show how this well-developed machinery facilitates the formal establishment, under a natural  scaling, of {\it fluid} and {\it diffusion limits} that are consistent with kinematic wave traffic flow theory. 
\item[$\circ$] The fluid limit provides accurate 
{\it approximations} for the means of the vehicle densities, while the diffusion limit can be used to approximate 
the corresponding correlations  (both in the spatial and temporal sense). Informally, these fluid and diffusion approximations can be seen as a law of large numbers and a central limit theorem, respectively, at the sample-path level. Our results allow for generic macroscopic fundamental diagrams (fulfilling a mild regularity assumption), and provide explicit expressions for means and (co-)variances that can be numerically evaluated in an efficient manner.
\item[$\circ$]
Whereas \cite{JL2012,JL2013} focus on vehicle densities only, we point out how diffusion results can be used to produce accurate approximations of the distribution of the {\it travel time}  experienced by vehicles moving through the network. 
\item[$\circ$] The accuracy of our approximation procedure is demonstrated through a series of {\it numerical experiments}. In addition, we  show that our methodology is able to reproduce known traffic phenomena such as forward propagation, backward moving traffic jams and shockwave formation, with a per-phenomenon consistent estimate for the 
variances and covariances of traffic densities.
\end{itemize}

The results obtained in this paper can be used to numerically evaluate the impact of various traffic control mechanisms. For instance, as imposing a speed limit effectively means an adaptation of the fundamental diagram, we can assess the efficacy of such a measure. Importantly, the framework used is highly flexible, in that the fundamental diagram needs to {fulfill} only mild regularity assumptions. This flexibility can be exploited when studying the impact of technological developments on traffic dynamics. For instance, one can evaluate scenarios in which a certain fraction of the cars is self-driving; an obvious prerequisite is the availability of (a proxy of) the corresponding fundamental diagram.

\vspace{3mm}

This paper is organized as follows. So as to shed light on the properties stochastic traffic flow models should {fulfill}, we first provide in Section \ref{KWM} a brief account of kinematic wave models. We pay attention to this class of models because, while being essentially standard for many readers with a background in transportation research, researchers with a background in operations research and applied probability may be less familiar with it. Building on kinematic wave models, we define in Section \ref{MD} our multi-class counterpart of the discrete-space model introduced in \cite{JL2012,JL2013}. Importantly, the setup proposed is highly flexible, and thus covers a {wide} variety of macroscopic fundamental diagrams proposed in the literature. Then, in Section \ref{SL}, we define our scaling and prove the fluid and diffusion limit results for the vehicle densities, yielding a Gaussian approximation for the corresponding distribution. Then it is shown in Section \ref{TT} how related diffusion results lead to an approximation for the travel-time distribution. Numerical experiments are provided in Section \ref{NE}, showing the method's accuracy and its capability to reproduce well-known phenomena. Section \ref{EGR} presents an account of extensions, generalizations, and ramifications. Concluding remarks are presented in Section \ref{CR}. 

\section{Theory of Kinematic Wave Models}\label{KWM}

The stochastic traffic flow model, that will be the main object of study in this paper, will be defined in Section \ref{MD}. There we define our model in such a way that the mean dynamics obey a conservation law, {which} is the fundamental modeling principle in every kinematic wave traffic flow model. To provide the reader with the necessary background, we review in this section the basics of kinematic wave models and corresponding conservation laws.

\subsection{Kinematic Wave Models and Conservation Laws}

Kinematic wave models are a class of macroscopic traffic flow models that consider a road segment without any intermediate sources or sinks as a continuous space, and study the propagation of vehicle density (defined as the number of vehicles per unit length of road) over the road segment. The absence of sources and sinks implies that traffic density should be preserved over the segment, which {suggests} the use of conservation laws to describe the evolution of traffic density over the segment as {\it fluid}.

In our framework we consider $m$ different types of vehicles. Let, for $j\in\{1,\ldots, m\}$, the {type-$j$} vehicle density be given by $\rho_j(x,t)$ and the {type-$j$} velocity by $v_j(x,t)$, where $t$ denotes time and $x$  the position on the segment. Conservation of mass (or density) in an arbitrary part of the road segment $[x_1,x_2]$, in a time interval $[t_1,t_2]$, is expressed as
\begin{equation}\label{Eqn: discrete conservation}
\int_{x_1}^{x_2} \rho_j(x,t_2) \diff x - \int_{x_1}^{x_2} \rho_j(x,t_1) \diff x 
= \int_{t_1}^{t_2} \rho_j(x_1,t) v_j(x_1,t) \diff t - \int_{t_1}^{t_2} \rho_j(x_2,t) v_j(x_2,t) \diff t, 
\end{equation}
with $j\in\{1,\ldots,m\}$; cf.\ \cite[Section 2.1]{leveque1992}. Here it is implicit that traffic {moves} from $x_1$ to $x_2$.  

Now, in traffic flow theory there is the fundamental, and empirically backed assumption, that the velocity $v_j(x,t)$ is a function of {\it all}\, the type  densities $\rho_k(x,t)$, where $k\in\{1,\ldots,m\}$. In the literature, this functional {relationship} is referred to as the (\textit{macroscopic}) \textit{fundamental diagram} (MFD),  usually expressed as 
\[
	q_j(\rho(x,t)) = \rho_j(x,t)\, v_j(\rho(x,t)),
\]
where $q_j$ is the flow of {type-$j$} density, and $\rho$ is the length-$m$ vector of traffic densities. Combining \eqref{Eqn: discrete conservation} with the MFD allows the derivation  \cite[Section 2]{leveque1992} of the following system of conservation laws:
\begin{equation}\label{Eqn: system conservation laws}
	 \frac{\partial \rho_j(x,t)}{\partial t} + \frac{\partial q_j(\rho)}{\partial \rho} \frac{\partial \rho(x,t)}{\partial x} = 0,\hspace{12mm}
		 \rho_j(x,0) = \rho_{j,0}(x),
\end{equation}
where ${\partial q_j(\rho)}/{\partial \rho}$ is to be interpreted as the gradient of $q_j$ and $\rho_{j,0}(x)$ is some initial vehicle density (as a function of the position $x$).

\subsection{The Macroscopic Fundamental Diagram}
As mentioned in the introduction, the fundamental diagram describes flow as a function of the vehicle density. 
The generally accepted view is that flow is an (approximately) concave function of the overall density, which is zero on the boundaries of some compact domain and positive on its interior. The reasoning behind this is simple: when traffic density is low, vehicles are not obstructed and can thus drive at maximum speed. As the density increases, vehicles will be obstructed to a greater extent due to congestion, thus lowering their velocity. Ultimately, if the traffic density becomes too {high}, then vehicles do not have enough space in front of them, resulting in zero velocity and flow.

The first fundamental diagram, for single-class traffic, can be found in a pioneering study by Greenshields \cite{green}. Since then, researchers have tried to find functional forms that reproduce various traffic phenomena as well as possible; see for more background  \cite[Section 5]{maerivoet2005traffic}. For single-class traffic, three notable MFDs can be found in \cite{daganzo1994cell,drake,smul}. In recent years, to capture effects due to mixed vehicle traffic, multi-class MFDs have been proposed; see{, e.g.,~}\cite{BC, cb2003, li2008,NGO,ww2002,ZHA}. 

\subsection{Solutions for Conservation Laws, Relation with  Stochastic Traffic Flow Model}

The functional forms for $\rho$ and $q$ that are used in practice are typically not differentiable, so that the problem in \eqref{Eqn: system conservation laws} has no strong solutions. One therefore resorts to weak solutions, satisfying additional conditions ({e.g.,} \textit{Rankine-Hugoniot} and \textit{entropy conditions}) to ensure uniqueness and correct physical behavior. For a mathematical treatment of conservation laws see, for instance, \cite[Section 3.4 and Ch.\ 11]{evans10}. 
Whereas it is often possible to show existence of weak solutions of \eqref{Eqn: system conservation laws}, they rarely have explicit analytical expressions. Therefore, for practical applications, one usually relies on numerical methods. In traffic flow theory the most widely used method is the {\it Godunov method}, \cite[Section 13]{leveque1992}, which has many attractive features, assuming that the Courant-Friedrichs-Lewy (CFL) condition (\cite[Eqn.\ (10.55), Eqn.\ (13.11)]{leveque1992}) is satisfied; see \cite[paragraph above Eqn.\ (18)]{cb2003}.

The Godunov method exactly solves \eqref{Eqn: system conservation laws} with an approximate, piece-wise constant, initial condition. The $x$-axis, which is the road segment in our case, is discretized into $d$ cells (which can be done according to a variable mesh), and the initial density is taken constant in each cell, by integrating the ($m$-dimensional) initial condition $\rho(x,0)$ between the boundaries of each cell. This defines a set of $d-1$ Riemann problems, the solution of which describes the flow between each of the cell boundaries exactly. As such, one can describe the exact evolution of the system with the equation, for $i\in\{2,\ldots,d\}$,
\begin{equation}\label{Eqn: Godunov eqn}
	\rho(x_i,t+\Delta t) =  \rho(x_i,t) + \frac{\Delta t}{\Delta x_i} \left( q_{i-1} - q_{i} \right),
\end{equation}
where $\rho(x_i,t)$ is now an $m$-dimensional vector of densities at time $t$ in the $i$-th cell, $\Delta t$ and $\Delta x_i$ {are, respectively, the lengths of the discrete time-steps and the mesh-width of cell~$i$}, and $q_{i}$ is the ($m$-dimensional) vector of flows on the boundary from cell~$i$ to cell~$i+1$, given by the solution of the corresponding Riemann problem. 
To be able to use the Godunov scheme, the challenge is to give a complete description of  \textit{discrete flux-function} $q_{i}$. In {general,} there is no explicit solution for the associated Riemann problem, but there are conditions under which a solution is guaranteed for small time steps \cite[Thm.\ 4 in Section 11]{evans10}. 

We finish this section by mentioning that an approximate solution of \eqref{Eqn: system conservation laws}, using the Godunov method, is also known as the {\it cell transmission model} (CTM). This concept, due to Daganzo \cite{daganzo1994cell}, relates to a discrete-time, discrete-space traffic flow model. One can show that \eqref{Eqn: system conservation laws} can be recovered from the CTM  by letting the cell lengths go to $0$ \cite{JL2012}. In our paper, we take an arbitrary (possibly multi-class) MFD as our starting point. We use the associated discrete flux-function to define a stochastic traffic flow model, such that the mean dynamics are given by the CTM. In {\autoref{SL},} we approximate the variance of this process through a diffusion limit, which allows for studying the distributions of vehicle densities in various types of dynamics, as will be illustrated in the numerical experiments of \autoref{NE}. 

\section{Model Definition}\label{MD}
In this section we introduce the stochastic traffic flow model as a strong Markov process, for which we subsequently define the corresponding state space and the infinitesimal generator. To keep the exposition as clear as possible, in this paper we primarily concentrate on a single road segment. However, the model, the underlying mathematical theory, and scaling results we present for this road segment setup carry over to a general network setting, as will be pointed out in \autoref{EGR} (along with various other extensions). We conclude the section by giving two relevant examples of kinematic wave models which we extend into stochastic traffic flow models, and which will be used later in the paper when numerically illustrating our results.

\subsection{State Space}
Consider a road segment without any intermediate sources or sinks, on which $m$~different types of vehicles are defined. We divide the road in $d$~cells, where $\ell_i$ is the length of cell $i$, for $i\in\{1,\ldots,d\}$. We denote by $X_{ij}(t)$ the number of {type-$j$} vehicles, for $j\in\{1,\ldots,m\}$, that are in cell~$i$ at time~$t \in [0,\infty)$. In the context of our road-traffic {model,} it is evident that the number of {type-$j$} vehicles in cell~$i$ has an upper bound, which we denote $X^{\rm jam}_{ij}$. As a consequence, $X_{ij}(t) \in \{0,1,\ldots,X^{\rm jam}_{ij}\}$. In the {sequel,} we denote $X(t)$ for the $md$-dimensional random vector with entries $X_{ij}(t)$.

Equivalently to $X_{ij}(t)$, we can consider the {type-$j$} vehicle density in cell~$i$ at time~$t$, which we denote $\rho_{ij}(t)$. It is defined as
\begin{equation}\label{Eqn: density <-> number of cars}
	\rho_{ij}(t) := \frac{X_{ij}(t)}{\ell_i},
\end{equation}
taking values in $\{0,1/\ell_i,2/\ell_i,\ldots,X_{ij}^{\rm jam}/\ell_i\}$. In the literature $\rho_{ij}^{\rm jam} := X^{\rm jam}_{ij} / \ell_i$  is referred to as the \textit{jamming density}  (of {type-$j$} vehicles in cell~$i$, that is).

\subsection{Transition Rates}

To describe the flows of {type-$j$} vehicles between the cells, we assume the existence of a discrete flux-function, obtained from a general (possibly multi-class) MFD {by solving the Riemann problem}. We proceed by formally introducing such a discrete flux-function, describing the flow from cell $i$ to cell $i+1$ for each of the types $j$. Denote $q^{{\rm max}}_{ij}$ for the largest possible flow of {type-$j$} vehicles between the neighboring cells $i$ and $i+1$. The assumption below says that the flow from cell $i$ to cell $i+1$ is a function of the densities in both cells, i.e., of $(\rho_{ik})_{k}$ and $(\rho_{i+1,k})_{k}$, where the first and last cell have to be handled slightly differently. Moreover, the function is dependent on $i$, {i.e.,} on the position of a cell, to model cell-dependent MFDs. This is particularly useful if one wishes to incorporate heterogeneity between the cells due to{, e.g.,} differences in road geometry and maximum velocity.
\begin{assumption}[Macroscopic Fundamental Diagram]\label{Assumption: MFD}
For $i \in \{1,\ldots,d-1\}$, we assume that there is a function 
\begin{equation}\label{Eqn: general MFD}
	\tilde q_i : \bigtimes_{j=1}^m \left( [0,\rho^{{\rm jam}}_{ij}] \times [0,\rho_{i+1,j}^{{\rm jam}}] \right) \to \bigtimes_{j=1}^m  [0,q^{{\rm max}}_{ij}],
\end{equation}
which is Lipschitz-continuous with Lipschitz constant $K_i \in (0,\infty)$. 
We also assume that there is a function 
\begin{equation}\label{Eqn: general MFD_0}
	\tilde q_0 : \bigtimes_{j=1}^m \left( [0,\infty) \times [0,{\rho_{1j}^{{\rm jam}}}] \right) \to \bigtimes_{j=1}^m  [0,q^{{\rm max}}_{0j}],
\end{equation}
which is Lipschitz-continuous with Lipschitz constant $K_0 \in (0,\infty)$; likewise, we assume that there is a function 
\begin{equation}\label{Eqn: general MFD_d}
	\tilde q_d : \bigtimes_{j=1}^m \left( [0,\rho_{dj}^{{\rm jam}}]  \times [0,\infty) \right) \to \bigtimes_{j=1}^m  [0,q^{{\rm max}}_{dj}],
\end{equation}
which is Lipschitz-continuous with Lipschitz constant $K_d \in (0,\infty)$.
It is, in addition, assumed that the $j$-th component of $\tilde q_i$ is $0$ whenever, for $j\in\{1,\ldots,m\}$, the $j$-th argument is $0$; all components of $q_i$ are $0$ whenever,  for some $j\in\{1,\ldots,m\}$, the $(m+j)$-th argument is $\rho^{{\rm jam}}_{i+1,j}$. 
\end{assumption}
\noindent 
The Lipschitz property is natural to assume, since by considering the density as a continuum, it entails that the flow should be a continuous function of the density, and that the change in flow should be bounded when adding or removing a vehicle from a cell. The last part of the assumption ensures that if there are no cars in the origin cell then there can be no flow between the two cells involved; likewise if the destination cell is full. 

Using \autoref{Assumption: MFD}, we proceed by defining the infinitesimal generator that makes $\{X(t), t \geq 0\}$, or equivalently $\{\rho(t),t\geq 0\}$, a Markov process. Consider two neighboring cells~$i$ and $i+1$, and let the current density be $\rho$.  For $i \in \{1,\ldots,d-1\}$ and $j \in \{1,\ldots,m\}$, we let the time until a car of type~$j$ passes the boundary between cells~$i$ and $i+1$ be an exponentially distributed random variable with rate
\[
	q_{i,j}(\rho) := \tilde q_{i,j} \big((\rho_{ik})_{k},(\rho_{i+1,k})_{k} \big).
\]

To model arrivals at cell~1, we introduce independent Poisson processes with rate~$\lambda_j$, for  $j\in\{1,\ldots,m\}$. We let the time until a {type-$j$} vehicle enters at cell~1, for $j\in\{1,\ldots,m\}$,  be an exponentially distributed random variable with rate
\[
	q_{0,j}(\rho) := \min\left\{\lambda_j ,\sup_{x \in (\R^+)^{m}} \tilde q_{0,j} \left( x, (\rho_{1k})_k \right)\right\}.
\]
The interpretation of this formula is that cars of type~$j$ constantly arrive at rate~$\lambda_j > 0$, but that the flow  is bounded {due to} the {vehicles} already present in cell 1. 
Similarly, we let the time until a {type-$j$} vehicle leaves cell~$d$, for $j\in\{1,\ldots,m\}$, be an exponentially distributed random variable with rate
\[
	q_{d,j}(\rho) := \min\left\{\nu_j ,\sup_{x \in (\R^+)^{m}} \tilde q_{d,j} \left( (\rho_{dk})_k, x \right)\right\},
\]
where $\nu_j > 0$. The interpretation is that vehicles of type~$j$ depart at a rate that depends on the densities in cell~$d$, but that is bounded by some maximal outflow rate~$\nu_j$.
We assume that, given $\rho$, all the exponentially distributed times, as defined above, are independent.

From the above it follows that the evolution of the number of {type-$j$} vehicles in cell $i$, with $i\in\{1,\ldots,d\}$ and $j\in\{1,\ldots, m\}$ is given, for a given initial state $X(0)$, by
\begin{equation}\label{Eqn: dX}
	X_{ij}(t) = X_{ij}(0) + Y_{i-1,j}\left( \int_0^t q_{i-1,j}(\rho(s)) \diff s \right)- Y_{i,j}\left( \int_0^t q_{i,j}(\rho(s)) \diff s \right),
\end{equation}
where $Y_{i,j}(t)$ are independent unit-rate Poisson processes for $i\in\{0,\ldots,d\}$ and $j\in\{1,\ldots, m\}$. {To} see this, note that the Poisson processes $Y_{i,j}$ are counting how many cars of type~$j$ pass between the boundary between cells $i$ and $i+1$. To measure the impact on the density, we scale these processes so that they have jumps of size $1 / \ell_i$. By utilizing \eqref{Eqn: density <-> number of cars}, we have the analogous process
\begin{equation}\label{Eqn: drho}
		\rho_{ij}(t) = \rho_{ij}(0) + \frac{1}{\ell_i} Y_{i-1,j}\left( \int_0^t q_{i-1,j}(\rho(s)) \diff s \right) 
		- \frac{1}{\ell_i} Y_{i,j}\left( \int_0^t q_{i,j}(\rho(s)) \diff s \right).
\end{equation}
We have thus constructed the $dm$-dimensional Markov process $\{\rho(t), t \geq 0\}$.

\subsection{Example MFDs}
\label{sec: example MFDs}

To illustrate the flexibility of  \autoref{Assumption: MFD}, we now give two concrete examples of MFDs and their associated discrete flux-functions $\tilde q$. In \autoref{NE} we will use these examples.

\begin{example}[Daganzo MFD]\label{Example: Daganzo}{\em 
Our first example is a discrete flux-function for a single-class CTM that was proposed in \cite{daganzo1994cell}. The flux-function describes the flow between two neighboring cells $i$ and $i+1$. Since traffic flows in one direction, cell~$i$ will be sending traffic and cell~$i+1$ will be receiving traffic. We define 
the \textit{sending} function $S(\cdot)$ and the \textit{receiving} function $R(\cdot)$ by
\[
	S(\rho) := v^f \rho \wedge q^{{\rm max}}, \quad
	 R(\rho) := w(\rho^{{\rm jam}} - \rho) \wedge q^{{\rm max}}.
\]
Here $v^f$ is the velocity at which vehicles drive when they do not experience congestion (called the \textit{free-flow} regime), and $w$ is the velocity at which traffic jams move upstream. These specific functions are based on the assumption that vehicles drive with their maximum velocity $v^f$ when vehicle density is low, whereas the flow {decreases linearly when congestion occurs}. The flow between the cells is now given by
\[
	\tilde q^{\textrm{D}}(\rho_i, \rho_{i+1}) = \min \{ S(\rho_i),R(\rho_{i+1}) \},
\]
where $\rho_i$ is the density of cell $i$. In Figure~\ref{Fig: Daganzo MFD} we have plotted an example with $v^f = 60$ km/h, $w = 12$ km/h, $q^{{\rm max}} = 1800$ veh/h and $\rho^{{\rm jam}} = 180$ veh/km, with the MFD on the left, and the sending and receiving functions that characterize the discrete flux-function on the right.

It is not hard to check that $\tilde q^{\rm D}$ satisfies \autoref{Assumption: MFD}. As both $S(\cdot)$ and $R(\cdot)$ are Lipschitz, the minimum of $S(\cdot)$ and $R(\cdot)$ is Lipschitz as well. Moreover, when $\rho_i = 0$ or $\rho_{i+1} = \rho^{{\rm jam}}$, then $\tilde q^{\textrm{D}}$ equals zero.

The paradigm of a sending and receiving function can be generalized to obtain other single-class discrete flux-functions. More precisely, the functions $S(\cdot)$ and $R(\cdot)$ can be every non-negative, Lipschitz continuous function with compact support. Moreover, it can be shown that every single-class MFD has a discrete flux-function, obtained by solving the Riemann problem associated to the Godunov scheme, {which} can be written as the minimum of a sending and receiving function.


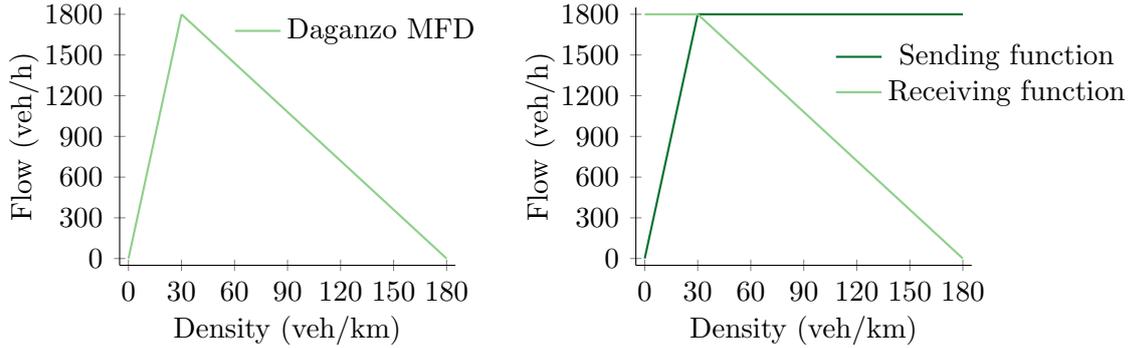
\begin{figure}
\begin{tabular}{cc}
\begin{tikzpicture}	
\begin{axis}[
    xlabel={Density (veh/km)},    ylabel={Flow (veh/h)},
    xmin=-5, xmax=185,    ymin=-50, ymax=1850,
    xtick={0,30,60,90,120,150,180},    ytick={0,300,600,900,1200,1500,1800},
    xticklabel style={/pgf/number format/.cd,fixed,set thousands separator={}},
    yticklabel style={/pgf/number format/.cd,fixed,set thousands separator={}},
	scaled x ticks=false,
	scaled y ticks=false,
    xtick pos=left,    ytick pos=left,    axis lines=left,
    width = 6cm,    height = 5cm,
    no markers,    x axis line style=-,    y axis line style=-,
    y label style={at={(axis description cs:0.07,.5)},anchor=south},
    font = \small,
    legend style={at={(1.1,1)},anchor=north east, draw=none}
]
\addplot+[thick,solid,mark=none,color32] coordinates {(0,0) (30,1800)};
\addplot+[thick,solid,mark=none,color32] coordinates {(180,0) (30,1800)};
\legend{Daganzo MFD}
\end{axis}
\end{tikzpicture} &
\begin{tikzpicture}	
\begin{axis}[
    xlabel={Density (veh/km)},    ylabel={Flow (veh/h)},
    xmin=-5, xmax=185,    ymin=-50, ymax=1850,
    xtick={0,30,60,90,120,150,180},    ytick={0,300,600,900,1200,1500,1800},
    xticklabel style={/pgf/number format/.cd,fixed,set thousands separator={}},
    yticklabel style={/pgf/number format/.cd,fixed,set thousands separator={}},
	scaled x ticks=false,
	scaled y ticks=false,
    xtick pos=left,    ytick pos=left,    axis lines=left,
    width = 6cm,    height = 5cm,
    no markers,    x axis line style=-,    y axis line style=-,
    y label style={at={(axis description cs:0.07,.5)},anchor=south},
    font = \small,
    legend style={at={(1.5,0.9)},anchor=north east, draw=none}
]
\addplot+[thick,solid,mark=none,color35] coordinates {(0,0) (30,1800)};
\addplot+[thick,solid,mark=none,color32] coordinates {(0,1800) (30,1800)};
\addplot+[thick,solid,mark=none,color35] coordinates {(180,1800) (30,1800)};
\addplot+[thick,solid,mark=none,color32] coordinates {(180,0) (30,1800)};
\legend{Sending function, Receiving function}
\end{axis}
\end{tikzpicture}
\end{tabular}
\caption{Daganzo MFD (left) and corresponding sending and receiving function, determining the associated discrete flux-function (right). }
\label{Fig: Daganzo MFD}
\end{figure}

}
\end{example}

\begin{example}[Chanut and Buisson MFD]\label{Example: CB}{\em 
This MFD was proposed in \cite{cb2003} for a multi-lane, two-class traffic flow model, and is a generalization of Smulders' fundamental diagram \cite{smul}. In the model, {vehicle} classes are differentiated by velocities and lengths, and therefore they influence the overall density and flow differently. We first present the MFD for the kinematic wave model, and present the discrete flux-function {afterward}.

For simplicity we refer to the first class of vehicles as \textit{cars}, and to the second class of vehicles as \textit{trucks}. In the model, the jamming density is given by
\[
	\rho^{{\rm jam}}(\rho_1,\rho_2) : = \frac{N}{\rho_1 L_1 + \rho_2 L_2} (\rho_1 + \rho_2),
\]
where $\rho_1,\rho_2$ are the car and truck densities, $L_1$ and $L_2$ are the respective vehicle lengths, and $N$ is the number of lanes. {The critical density, which distinguishes the free-flow regime and the congestion regime, and which is parameterized by $\beta \in [0.2,0.5]$, is given by}
\[
	\rho^{\textrm{c}}(\rho_1,\rho_2) := \beta \rho^{{\rm jam}}(\rho_1,\rho_2).
\]

In the free-flow regime, defined as $\rho_1 + \rho_2 \leqslant\rho^{\textrm{c}}(\rho_1,\rho_2)$, the vehicle velocity starts at the maximum velocity $v^f_i$ when $\rho_1 + \rho_2 = 0$, and decreases linearly in $\rho := \rho_1 + e \rho_2$, with $e := L_2/L_1$, until the critical velocity $v^{\textrm{c}}$ is reached (at the critical density $\rho^{\rm c}$). In congestion, both vehicle classes have the same velocity, and it is assumed that the total flow decreases linearly in $\rho$. The quantity $\rho$ can be interpreted as the density in terms of passenger-car equivalents (PCEs); one truck is equivalent to $e$ cars. In the congested regime, the flow of vehicles, measured in PCEs, can thus be expressed as
\[
	q_{\textrm{pce}}(\rho_1,\rho_2) := C\frac{\rho^{{\rm jam}}(\rho_1,\rho_2)  - (\rho_1 + \rho_2)}{\rho^{{\rm jam}}(\rho_1,\rho_2) - \rho^{\textrm{c}}(\rho_1,\rho_2) },
\]
where $C := v^\textrm{c} \rho^{\textrm{c}}(\rho_1,0)$. If $\rho_1 + \rho_2 \leq \rho^{\textrm{c}}(\rho_1,\rho_2)$ (in free-flow, that is), then the flow of both vehicle classes is now given by
\begin{equation*}
	\begin{split}
		& q_1 = \rho_1 v_1, \quad v_1 = v^f_1 - (v^f_1 - v^\textrm{c})\frac{\rho_1 + \rho_2}{\rho^{\textrm{c}}(\rho_1,\rho_2)}; \\
		& q_2 = \rho_2 v_2, \quad v_2 = v^f_2 - (v^f_2 - v^\textrm{c})\frac{\rho_1 + \rho_2}{\rho^{\textrm{c}}(\rho_1,\rho_2)},
	\end{split},
\end{equation*} 
whereas if  $\rho_1 + \rho_2 > \rho^{\textrm{c}}(\rho_1,\rho_2)$ (in congestion, that is), then
\begin{equation*}
	q_1 = \rho_1 v,\:\:\:\: q_2 = \rho_2 v,\:\:\:\: v = \frac{q_{\textrm{pce}}(\rho_1,\rho_2)}{\rho_1 + e \rho_2}.
\end{equation*}

To use this MFD in a discrete model, one has to determine the associated discrete flux-function, by solving the Riemann problem for every possible combination of two neighboring cell densities. This discrete flux-function is given in \cite{cb2003}, and can be written in terms of the sent density and the received density, as in Example~\ref{Example: Daganzo}, except for the case where the upstream cell is in the free-flow regime, the downstream cell is in the congested regime and the shock discontinuity propagates at a negative speed. 

We conclude this example by summarizing the discrete flux-function. Let $\Delta$ and $\Omega$ denote the sent and received density, respectively. If a cell is in the free-flow regime, then they are given by
\[
\begin{aligned}
	 \Delta = \Delta_1 + e \Delta_2,\quad
 \Omega = C,\quad 
 \Delta_i = \rho_i v_i, \quad i=1,2,
\end{aligned}
\]
whereas if a cell is in the congested regime, then
\[
\begin{aligned}
	 \Delta = \Delta_1 + e \Delta_2 C,\quad 
	 \Omega = (\rho_1 + e \rho_2)v, \quad
	 \Delta_i = \frac{\rho_i}{\rho_1 + e \rho_2} C, \quad i=1,2.
\end{aligned}
\]
Now, if for two neighboring cells, the upstream cell is in free-flow and the downstream cell is in congestion, then one has to consider the speed $s$ of the shock discontinuity, which is given by
\[
	s = \frac{(q_1^{\rm u} + e q_2^{\rm u}) - (q_1^{\rm d} + e q_2^{\rm d})}{(\rho_1^{\rm u} + e \rho_2^{\rm u}) - (\rho_1^{\rm d} + e \rho_2^{\rm d})},
\]
where the upper indexes `u' and `d' denote whether the quantities $q_i$ and $\rho_i$ belong to the upstream or downstream cell. If $s < 0$, then the sending-receiving concept does not apply, as the flow between the two cells is given by an intermediate state of densities, according to the solution of the Riemann problem; then we have
\[
\begin{aligned}
	& q_1 = v \frac{q_1^{\rm u} - s\rho_1^{\rm u}}{v - s}; 
	& q_2 = v \frac{q_2^{\rm u} - s\rho_2^{\rm u}}{v - s},
\end{aligned}
\]
where $v$ is the velocity of vehicles in the downstream cell.

In any other configuration of two neighboring states, {e.g.,} if either the upstream cell and downstream cell are both in the free-flow regime or if $s > 0$, then the flow is given by the following sending-receiving scheme: if $\Delta^{\rm u} \leq \Omega^{\rm d}$, then we have
\[
\begin{aligned}
	& q_1 = \Delta_1^{\rm u};
	& q_2 = \Delta_2^{\rm u},
\end{aligned}
\]
whereas if $\Delta^{\rm u} >\Omega^{\rm d}$, then
\[
\begin{aligned}
	& q_1 = \frac{\rho_1^{\rm u}}{\rho_1^{\rm u} + e \rho_2^{\rm u}} \Omega^{\rm d}; 
	& q_2 = \frac{\rho_2^{\rm u}}{\rho_1^{\rm u} + e \rho_2^{\rm u}} \Omega^{\rm d}.
\end{aligned}
\]}
\end{example}

\section{Scaling Limits}\label{SL}
\label{sec: scaling limits}

Now that we have defined our model, we proceed by analyzing it. As exact analysis is infeasible, we do so under a certain scaling. Concretely, we show that after an appropriate centering and normalization, the $md$-dimensional process $\rho(\cdot)$ converges to a Gaussian process as the scaling parameter $n$ goes to $\infty.$ In this way, we can find an approximation for the distribution of $\left(\rho(t_1),\dots,\rho(t_k)\right)$ at times $0 \leq t_1 < \ldots < t_k < \infty$ (corresponding to a $mdk$-dimensional normally distributed random vector). 
The  scaling we consider concerns both the lengths of the cells and time, i.e., we work with $\ell_i \mapsto n \ell_i$ for $i\in\{1,\ldots,d\}$, and $t \mapsto nt$. We 
thus obtain a sequence of processes $\{\rho^n(t)\}_n$. 

A first main result of this section is that $\{\rho^n(\cdot)\}_n$ converges almost surely, as $n \to \infty$, to a limiting process $\bar{\rho}(\cdot)$. This convergence is usually referred to as a \textit{first order approximation}, or a \textit{fluid limit}, and can be seen as a law of large numbers at the path level. This fluid limit is consistent with the CTM associated to the function we choose in \autoref{Assumption: MFD}, as its dynamics coincide with~\eqref{Eqn: Godunov eqn}. 
As a second main result of this section, we  prove that the sequence of processes $\hat{\rho}^n(\cdot) := \sqrt{n}(\rho^n(\cdot) - \bar{\rho}(\cdot))$ converges in distribution, as $n \to \infty$, to a Gaussian process, which is a \textit{second order approximation} or a \textit{diffusion limit}. This result is essentially a central limit theorem at the path level.

The scaling that we impose can intuitively be understood as follows. For our results we want to exploit principles underlying the classical law of large numbers and central limit theorem, analogously to when one considers a large sample of (approximately independent and identically distributed) random variables. Hence, we need a large number of vehicles to be present in each of the cells. To achieve this, we scale the length of the cells by a factor of $n$, where we let $n$ grow large. However, the discrete flux-functions defined in \eqref{Eqn: general MFD}-\eqref{Eqn: general MFD_d} have a compact co-domain. This means that, to keep the flow across cells invariant under the scaling of the cell lengths, 
we need to scale the number of vehicles that jump per unit time, or equivalently, scale time. Put differently, without scaling these fluxes, we have the undesired effect that the processes 
\[
	\frac{1}{n \ell_i} Y_{i,j}\left( \int_0^t q_{i,j}(\rho(s)) \diff s \right)
\]
converge almost surely to the zero process as $n \to \infty$. We observe that, in order to get a non-degenerate limiting process, leaving the flux over finite time intervals invariant, we should scale time by a factor $n$.

In this section, we first state and prove the fluid limit, and then move to the diffusion limit. In the sequel we denote by $\rho^n(t)$ the process $\rho(nt)$, where the lengths of cells are $n \ell_i$, for $i\in\{1,\ldots,d\}$. As a closing comment, all results in this section can be translated into results for $X(\cdot)$, due to \eqref{Eqn: density <-> number of cars}.

\subsection{Fluid limit}

In order to prove both the first and second order approximations for $\rho(t)$ (and hence for $X(t)$), we use two theorems from \cite[Ch.\ 8]{kurtz1981approximation}. These two theorems are written in terms of population processes, but directly translate to the setting we consider. To illustrate this, consider the scaled processes $\rho^n(t)$. By \eqref{Eqn: drho} we have 
\begin{align*}
	\rho^n_{ij}(t) = & \; \rho^n_{ij}(0) + \frac{1}{n \ell_i} Y_{i-1,j}\left( \int_0^{nt} q_{i-1,j}\left( \frac{\rho(s)}{n} \right) \diff s \right)  - \frac{1}{n \ell_i} Y_{i,j}\left( \int_0^{nt} q_{i,j}\left( \frac{\rho(s)}{n} \right) \diff s \right) \\
		= & \; \rho^n_{ij}(0) + \frac{1}{n \ell_i} Y_{i-1,j}\left( n \int_0^{t} q_{i-1,j}\left(  \rho^n(s) \right) \diff s \right)  - \frac{1}{n \ell_i} Y_{i,j}\left( n \int_0^{t} q_{i,j}\left( \rho^n(s) \right) \diff s \right),
\end{align*}
where the second equality is due to a change variables (with $s$ being replaced by $ns$), and the second line agrees with \cite[Eqn. (8.6)]{kurtz1981approximation}. Hence, we only need to verify the assumptions of \cite[Thm.\ 8.1]{kurtz1981approximation} so as to establish the fluid limit. We state the theorem, and verify these assumptions in the proof. Here, `u.o.c.' means `uniformly on compacts'. 

\begin{theorem}[Fluid limit]\label{Thm: Fluid limit cells}
	Under \autoref{Assumption: MFD} and $\lim_{n \to \infty} \rho^n(0) = \bar{\rho}(0)$ for some $\bar{\rho}(0)$, we have
\[
	\rho^n(\cdot) \overset{\rm a.s.}{\to} \bar{\rho}(\cdot), \quad u.o.c.,
\]
where, for $j\in\{1,\ldots,m\}$ and $i\in\{1,\ldots,d\}$, the process $\bar{\rho}(\cdot)$ satisfies the integral equation
\begin{equation}\label{Eqn: fluid limit from, for (ij)}
\bar{\rho}_{ij}(t) = \bar{\rho}_{ij}(0) + \frac{1}{\ell_i} \int_0^t q_{i-1,j}(\bar{\rho}(s)) \diff s - \frac{1}{\ell_i}  \int_0^t q_{i,j}(\bar{\rho}(s)) \diff s.
\end{equation}
\end{theorem}
\begin{proof}
As said before, we want to appeal to \cite[Thm.\ 8.1]{kurtz1981approximation}. Therefore, we verify the assumptions of this theorem. Because, by \autoref{Assumption: MFD}, the functions $\tilde q_i$ are Lipschitz on a compact domain, $\sup_{x} q_{ij}(x)$ exists, for every $i \in \{0,\ldots,d\}$ and $j \in \{1,\ldots,m\}$. Hence, by the fact that there is a finite number of possible transitions, the first assumption \cite[Eqn.~(8.10)]{kurtz1981approximation} directly follows:
\[
	\sum_{i=0}^d \sum_{j=1}^m \sup_{x \in \R^{dm}_+}q_{ij}(x) \leq dm \max_{\substack{i \in \{0,\ldots,d\} \\ j\in \{1,\ldots,m\}}} \sup_{x \in \R^{dm}_+}q_{ij}(x) < \infty
\] 
The other assumption, {i.e.,} \cite[Eqn. (8.11)]{kurtz1981approximation}, is written in our setting as
\[
	\bigg|\! \bigg| \sum_{i=1}^d \sum_{j=1}^m \frac{1}{\ell_i} \left[ \left(q_{i-1,j}(\rho_1) - q_{ij}(\rho_1)\right) - \left(q_{i-1,j}(\rho_2) - q_{ij}(\rho_2)\right)  \right] e_{ij} \bigg|\! \bigg| \leq K |\!|{\rho_1 - \rho_2}|\!|,
\]
for all $\rho_1, \rho_2$ in the state space of $\rho(\cdot)$. This property is an immediate consequence of the Lipschitz assumption that we imposed {on $\tilde q_i$}.
\end{proof}

\begin{remark}{\em 
We explain how the integral equation \eqref{Eqn: fluid limit from, for (ij)} is consistent with the Godunov-based approximate solution of \eqref{Eqn: system conservation laws}. Consider the time interval $[0,t]$, and take a partition $\pi$, such that $0 = t_0 < t_1 < \ldots < t_k = t$, and so that the mesh $\norm{\pi} := \sup_{l\in\{1,\ldots, k\}} |t_l - t_{l-1}| \leq \Delta t$, where $\Delta t$ follows from the so-called CFL condition (cf.\ \cite[Eqn.\ (13.11)]{leveque1992}) that makes the Godunov scheme stable. Using \eqref{Eqn: Godunov eqn} one can write
\begin{align*}
\rho^G_{ij}(t) & = \rho^G_{ij}(t_k-1) + \frac{t_k - t_{k-1}}{\ell_i} \left[ q_{i-1,j}(\rho(t_{k-1})) - q_{i,j}(\rho(t_{k-1})) \right] \\
& = \rho^G_{ij}(0) + \frac{1}{\ell_i} \sum_{l=1}^k \left(t_l - t_{l-1}\right) \left[ q_{i-1,j}(\rho(t_{l-1})) - q_{i,j}(\rho(t_{l-1})) \right],
\end{align*}
where $\rho^G_{ij}(t)$ is now the {type-$j$} density in cell $i$, given by the solution of the Godunov method, as emphasized by the superscript $G$.
Now, we interpret $\rho_{ij}^G(t)$ as a continuous function, by linearly interpolating between the values at the $t_l$, $l=1,\ldots,k$. Letting $k \to \infty$ such that $\norm{\pi} \to 0$, or instead letting $\Delta t \to 0$, we have that the above expression for $\rho_{ij}(t)$ converges to the expression for $\bar{\rho}_{ij}(t)$ in \eqref{Eqn: fluid limit from, for (ij)}, uniformly on compact time intervals.}
\end{remark}

To conclude, the results of this section show that the fluid limit is consistent with kinematic wave models. In Section \ref{sec: generating diff approx} we show that $\rhob(\cdot)$ is consistent with the mean dynamics of our model. The added value of our stochastic model, relative to existing deterministic models, lies in the fact that we also have a handle on the likelihood of fluctuations around the fluid limit,  using the diffusion limit that we establish in Section \ref{GDL}.

\subsection{Gaussian diffusion limit}\label{GDL}

We now turn to the diffusion limit by considering the scaled and centered process $\hat{\rho}^n$, defined as
\[
	\rhoh^n(t) := \sqrt{n} \left( \rho^n(t) - \rhob(t) \right).
\]
When establishing the diffusion limit, it is practical to work with a vector representation of the processes involved, for which we introduce the notation
\[
	F(\rho(t)) := \sum_{i=1}^{d} \frac{1}{\ell_i} \sum_{j=1}^m \left( q_{i-1,j}(\rho(t)) - q_{i,j}(\rho(t))  \right)e_{ij},
\]
where $e_{ij} \in \Z^{dm}$ is the vector that has a one on the $(i,j)$-th coordinate, and zeros everywhere else. This immediately gives us the following vector representation for $\rhob(\cdot)$:
\begin{equation}\label{Eqn: vector integral eqn, rho_bar}
	\rhob(t) = \rhob(0) + \int_0^t F(\rhob(s)) \diff s.
\end{equation}

To formulate a weak limit for $\rhoh^n(t)$, we want to use \cite[Thm.\ 8.2]{kurtz1981approximation}. This theorem, however, requires the existence of $\partial F(x)$, the matrix of partial derivatives of $F$, as a continuous and bounded operator, for which the Lipschitz assumption that we imposed {on $\tilde q_i$} is not enough. In the proof of \cite[Thm.\ 8.2]{kurtz1981approximation}, the existence of $\partial F$ is used for two statements. The first of these statements is that the matrix-valued ordinary differential equation
\begin{equation}\label{Eqn: matrix ODE dF}
	\frac{\partial}{\partial t}\Phi(t,s) = \partial F(\rhob(t)) \,\Phi(t,s), \quad \Phi(s,s) = I,
\end{equation}
has a unique solution. The second statement is that the mapping $P \colon D_{\R^{\rm d}}[0,\infty) \to D_{\R^{\rm d}}[0,\infty)$, given by
\[
	P \xi(t) := \xi(t) + \int_0^t \Phi(t,s) \,\partial F(\bar{X}(s)) \,\xi(s) \diff s,
\]
is continuous in the Skorohod topology. With the use of standard arguments from differential equations, the result from \cite[Thm.\ 8.2]{kurtz1981approximation} remains true under the weaker condition that $\partial F$ exists in the weak sense, and is bounded, which follows from the Lipschitz property that we imposed {on $\tilde q_i$}. For completeness, we summarize these arguments.

Since the $\tilde q_i$ are Lipschitz, $F(\cdot)$ is Lipschitz as well, so  that $\partial F$ exists in the weak sense and its component functions are in $L^\infty$ \cite[Thm.\ 5.8.4 and its remark]{evans10}. Obviously, every weak derivative of $F(\cdot)$ equals the `strong' derivative of $F(\cdot)$ at every point where $F(\cdot)$ is differentiable. Moreover, as $F(\cdot)$ is Lipschitz, the integral equation \eqref{Eqn: vector integral eqn, rho_bar} has a unique strong solution, so that we can consider $t\mapsto \partial F(\rhob(t))$ as a (deterministic) function of $t$.

Combining both observations, we can (and will) from now on let  $\partial F$ be the matrix that is given by the partial derivatives of $F$, whenever they exist; whenever the corresponding partial derivative does not exist,
we use $K = \max_{k \in\{1,\ldots,  (d+1)m\}} \kappa_k$ instead, with $\kappa_{k}$ denoting the Lipschitz constant corresponding to the $k$-th coordinate of $F(\cdot)$. That is, $\partial F$ is the matrix of the partial derivatives of $F$ in the weak sense, and we emphasize that \eqref{Eqn: matrix ODE dF} should be interpreted accordingly from now on. We then have that the operator $t\mapsto\partial F (\rhob(t))$ is a deterministic, bounded, almost everywhere continuous function. Therefore we can apply Carath\'eodory's theorem \cite[Thm.\ 5.1]{hale1969} to find that there exists a solution to \eqref{Eqn: matrix ODE dF}, which is unique due to  \cite[Thm.\ 5.3]{hale1969}. Furthermore, as $\partial F$ is bounded, the continuity of the function $P$ still holds. (It is noted that  simply \textit{assuming} the condition of \cite[Thm.\ 8.2]{kurtz1981approximation} {will} not lead to a satisfactory solution, as various frequently used MFDs from the literature do not obey this property.)

We are now ready to state our version of \cite[Thm.\ 8.2]{kurtz1981approximation}, the validity of which follows from the above considerations. We let $Z^n(\cdot) \Rightarrow Z(\cdot)$ denote weak convergence of processes $Z^n(\cdot)$ to a process $Z(\cdot)$, in the space of c\`adl\`ag functions with the Skorohod topology.

\begin{theorem}[Diffusion approximation]\label{Thm: diffusion cells}
Under \autoref{Assumption: MFD} and  $\lim_{n \to \infty} \sqrt{n}\,|\rho^n(0) - \rho_0| = 0$ for some $\rho_0$, we have that $\rhoh^n(\cdot) \Rightarrow \rhoh(\cdot)$, where $\rhoh_{ij}(t)$, $i\in\{1,\ldots,d\}$, $j\in\{1,\ldots,m\}$, obeys the stochastic integral equation
\begin{align}\nonumber
	\rhoh_{ij}(t) = & \; \frac{1}{\ell_i} \int_0^t \nabla q_{i-1,j}(\rhob(s)) \cdot \rhoh(s) \diff s - \frac{1}{\ell_i} \int_0^t \nabla q_{i,j}(\rhob(s)) \cdot \rhoh(s) \diff s \\
& 	\frac{1}{\ell_i} \int_0^t \sqrt{q_{i-1,j}(\rhob(s))} \diff B_{i-1,j}(s) - \frac{1}{\ell_i} \int_0^t \sqrt{q_{i,j}(\rhob(s))} \diff B_{i,j}(s) ,\label{lsde}
\end{align}
where $\nabla q_{i-1,j}$ and $\nabla q_{i,j}$ are gradients that are to be interpreted in the weak sense, corresponding with our redefinition of $\partial F$, where $x \cdot y$ denotes the Euclidean inner product for vectors $x$ and $y$, and where the $B_{i,j}(\cdot)$ are independent standard Brownian motions.
\end{theorem}

The limiting process $\rhoh(\cdot)$ thus satisfies the
 \textit{linear stochastic integral equation} (\ref{lsde}). It is known that this class of stochastic integral equations allows an explicit solution, which  is a Gaussian process with a known mean and covariance \cite[Section 5.6]{karatzas2012brownian}. To  state these results, we first introduce some notation.

Let $Q(\rho(t))$ be a vector of length $(d+1)m$ with entries $q_{i-1,j}$, $i\in\{1,\ldots,d\}+1$ and $j\in\{1,\ldots,m\}$, ordered lexicographically, {i.e.,} $Q_{(i-1)m+j} = q_{i-1,j}$. Now, let $H$ be the $dm \times (d+1)m$ matrix with $H_{kl} := \mathbbm{1}_{\{k=l\}} - \mathbbm{1}_{\{k+m=l\}}$
for $k\in\{1,\ldots, dm\}$ and $l\in\{1,\ldots, (d+1)m\}$. In addition, $L$ is a $dm$-dimensional diagonal matrix,
with the $k$-th diagonal element being $1/\ell_i$ if $\lceil k/m \rceil = i$, for $k \in \{1,\ldots,dm\}$ and $i \in \{1,\ldots,d\}$.
It is readily checked that with this new notation we have $F(\rho(t)) = LH\,Q(\rho(t))$. 
The process $\rhoh(\cdot)$ thus satisfies the (vector-valued) stochastic differential equation
\[
	\diff \rhoh(t) = \partial F(\rhob(t)) \rhoh(t) \diff t + L H\, \Sigma(\rhob(t)) \diff B(t),
\]
where $B(t)$ is a length $(d+1)m$ vector of independent standard Brownian motions and where $\Sigma(\rhob(t)) $ is the $(d+1)m \times (d+1)m$ diagonal matrix with  the square roots of ${Q(\rhob(t))}$ on the diagonal. 
The corresponding  {mean vector} and {covariance matrix}  are defined as
\begin{equation*}
	 M(t) := \E[ \rhoh(t) ],\:\:\:\:
		 \Gamma(s,t):=\textrm{cov}\left(\rhoh(s),\rhoh(t)\right) = \E[(\rhoh(s) - M(s))(\rhoh(t)-M(t))^\top].		 
\end{equation*}
In addition, $V(t) := {\rm var}[\rhoh(t)]=\Gamma(t,t).$
As in \cite[Section 5.6, Problems 6.1, 6.2]{karatzas2012brownian}, with $\bar\Phi(s):=\Phi(s,0)$, these allow explicit expressions:
\begin{align}
		 M(t) &= \bar\Phi(t) \left[M(0) + \int_0^t \bar\Phi^{-1}(s) \diff s \right], \nonumber \\
		 \Gamma(t,s) &= \bar\Phi(s) \left[ V(0) + \int_0^{t \wedge s} \bar\Phi^{-1}(u) LH\,\Sigma(\rhob(u))  \left(\bar\Phi^{-1}(u) \,LH\,\Sigma(\rhob(u))\right)^\top \diff u \right] \bar\Phi^\top(t) \label{eqn: Cov rho_t, rho_s}.
\end{align}
Moreover, $M(t)$ and $V(t)$ solve the linear (matrix) differential equations
\begin{align}
	 \frac{{\rm d}M(t)}{{\rm d}t}& = \partial F(\rhob(t)) \,M(t),  \nonumber \\
		 \frac{{\rm d}V(t)}{{\rm d}t} &=  \partial F(\rhob(t)) V(t) + V(t) (\partial F(\rhob(t)))^\top + LH\,\Sigma(\rhob(t)) (LH\,\Sigma(\rhob(t)))^\top .  \label{eqn: dVar rho}
\end{align}

At the beginning of this section, we expressed that our objective was to find an approximation for the distribution of $\left( \rho({t_1}),\ldots,\rho({t_k}) \right)$ for  $0 \leq t_1 < t_2 < \ldots < t_k < \infty$. With the fluid limit $\rhob(\cdot)$ and diffusion limit $\rhoh(\cdot)$ we have succeeded in doing so. In Section \ref{NE} we present numerical examples that assess the accuracy of the resulting approximation.


\begin{remark}{\em 
It is known that the accuracy of the diffusion approximation degrades in the vicinity of non-smooth points of the transition rate functions. The reason is that in such a setting the natural fluctuations of the process are such that it resides a non-negligible fraction on time on both sides of the non-smooth point, such that the local dynamics on both sides  matter. 
For more background on this phenomenon we refer to \cite{A1,A2}.}
\end{remark}

\section{Travel Times}\label{TT}
\label{sec: Travel Times}

In this {section,} we turn our focus to the analysis of the travel-time distribution of {type-$j$} vehicles. For $i\in\{1,\ldots,d\}$ and $k \in \{0,\ldots,d-i+1\}$, we define the {type-$j$} travel time, $T_{i,i+k,j}(t)$, to be the time that it takes a {type-$j$} vehicle to depart from cell $i+k$, {given that it resides in cell $i$ at time $t$}. If, for every $j \in \{1,\ldots,m\}$, {type-$j$} vehicles do not overtake other {type-$j$} vehicles within a cell, we have {(cf.\ \cite[Eqn.\ (40)]{Qian2017})}
\begin{equation}\label{Eqn: relation Y and T}
	\{T_{i,i+k,j}(t) > x\} = \{Y_{i+k,j}(t+x) < Y_{i-1,j}(t)\}, \quad x > 0.
\end{equation}
By this identity, a sequence of probabilities $\PP(T_{i,i+k,j}(t) > x_n)$, for $x_n >0$ and $n \in \{1,\ldots,N\}$, can be derived from the joint distribution of the random vectors $Y(t),{Y(t+x_1),}\ldots,Y(t + x_N)$. In this section, we approximate the joint distribution of these random vectors with a Gaussian distribution, using fluid and diffusion limits for $Y(\cdot)$, similar to the ones we derived in \autoref{sec: scaling limits}, under the same scaling. We thus find an approximation of the distribution of $T_{i,i+k,j}(t)$, for $i \in \{1,\ldots,d\}$, $k \in \{0,\ldots,d-i+1\}$ and $j \in \{1,\ldots,m\}$.

At first {sight,} it may look restrictive to require that {type-$j$} vehicles do not overtake each other within a cell. Noticing, however, that individual vehicles within a class are not systematically faster of slower than one another, it is anticipated that the approximation of the travel-time distribution is reasonably accurate. Moreover, when overtaking behavior within a class of vehicles occurs often, our setup allows the class to be split into two classes, say, a fast and a slow one.

Initially, one would think that the limit results for $Y(\cdot)$ directly follow from the ones we derived for $\rho(\cdot)$ in \autoref{sec: scaling limits}; with the vector $Y$ ordered lexicographically, we have
\begin{equation}\label{Eqn: relatie rho en Y}
	\rho(t) = L X(0) + L H\, Y(t),
\end{equation}
in accordance with \eqref{Eqn: drho}. The results for $\rho(\cdot)$, however, cannot be translated into results for $Y(\cdot)$, due to the easily verified fact the matrix $H$ is singular. A simple illustration of this is that when $Y(t)$ satisfies \eqref{Eqn: relatie rho en Y}, then so does $Y(t) + c$ for every $c \in \R$. 
As a consequence, we have to follow a different approach. 

Before we present the fluid and diffusion limits for $Y(\cdot)$, we make two final observations. First, the diffusion approximation can be established along the lines of the proof of \cite[Thm.\ 8.2]{kurtz1981approximation}, but we choose to give a different, more concrete proof, which illustrates how the redefinition of the operator $\partial F$ and the corresponding weak solutions naturally follow from taking the limit. Second, the point processes $Y_{i-1,j}(\cdot)$ have intensity $q_{i-1,j}(\rho(t))$ at time $t$, which by \eqref{Eqn: relatie rho en Y} is also a function of $Y(t)$. Moreover, as $L$ and $H$ have bounded norm, this function is clearly still Lipschitz in $Y(t)$. To simplify the notation, from now on we simply say that each $Y_{i-1,j}(t)$ has intensity given by $h(Y(t))$ at  time $t$, for a Lipschitz continuous function $h(\cdot)$. When we apply the scaling that we introduced in \autoref{sec: scaling limits}, we denote the intensity by $h^n(\cdot)$ to emphasize that the argument is reduced by a factor $n$.

We now present the first result, which is the counterpart of Thm.\ \ref{Thm: Fluid limit cells}.
\begin{proposition}\label{Prop: TT fluid limit}
	Consider the sequence of scaled processes $Y^n(\cdot)$, where  $Y^n(t) := \frac{1}{n} Y(nt)$. Under \autoref{Assumption: MFD} and if $X^n(0) \to \bar{X}(0)$, we have, as $n \to \infty$, that
\[
	Y^n(\cdot) \overset{\rm a.s.}{\to} \bar{Y}(\cdot), \quad u.o.c.\
\]
where $\bar{Y}(\cdot)$ satisfies
\[
	\bar{Y}(t) = \int_0^t h(\bar{Y}(s)) \diff s.
\]
\end{proposition}
\begin{proof}
	To prove the proposition, we can replicate the proof of \cite[Thm.\ 8.1]{kurtz1981approximation}. The conditions are met, as $h(\cdot)$ is Lipschitz in $Y(\cdot)$, and $h(\cdot)$ has compact support and is therefore uniformly bounded.
\end{proof}

We immediately present the diffusion limit for $Y(\cdot)$, which is the counterpart of Thm.~\ref{Thm: diffusion cells}. Here $\partial h$ is the weak derivative of $h$, defined similarly as $\partial F$ in \autoref{sec: scaling limits}. Also, $\bar\Sigma(\bar{Y}(t))$ is the $(d+1)m\times (d+1)m$ diagonal matrix with the square roots of $h(\bar{Y}(t))$ on the diagonal.

\begin{theorem}\label{Ydiff}
Consider the sequence of  centered and scaled processes $\{\hat Y^n(\cdot)\}_n$, where
$
	\hat{Y}^n(t) := \sqrt{n} \left( Y^n(t) - \bar{Y}(t) \right).
$
Under \autoref{Assumption: MFD} and if $\sqrt{n} \,| X^n(0) - \bar{X}(0) | \to 0$, we have, as $n\to\infty$, that \[\hat{Y}^n(\cdot) \Rightarrow \hat{Y}(\cdot),\] where $\hat{Y}(\cdot)$ satisfies
\[
	\diff \hat{Y}(t) = \partial h(\bar{Y}(t)) \hat{Y}(t) \diff t + \bar\Sigma(\bar{Y}(t))\, \diff {B}(t), \quad \hat{Y}(0) = 0,
\]
with $B(\cdot)$ a $(d+1)m$-dimensional vector of independent standard Brownian motions.\end{theorem}

\begin{proof}
	We begin by rewriting $\hat{Y}^n(\cdot)$ by adding and subtracting the compensator of $Y^n(\cdot)$,
	\begin{align*}
		\hat{Y}^n(t) & = n^{1/2} \left(Y^n(t) - \frac{1}{n} \int_0^{nt} h^n(Y(s)) \diff s + \frac{1}{n} \int_0^{nt} h^n(Y(s)) \diff s - \int_0^{t} h(\bar{Y}(s)) \diff s  \right) \\
		& = n^{-1/2} M^n(t) + \int_0^t \Psi^n(s) \hat{Y}^n(s) \diff s,
	\end{align*}
where we define the sequences of martingales $\{M^n(\cdot)\}_n$ and processes $\{\Psi^n(\cdot)\}_n$ by
\begin{equation*}
	M^n(t) := Y(nt) - \int_0^{nt} h^n(Y(s)) \diff s,  \quad \Psi_{kl}^n(t) :=  \frac{h_k(Y^n(s)) - h_k(\bar{Y}(s))}{ Y_l^n(s) - \bar{Y_l}(s) },
\end{equation*}
where $k,l \in \{1,\ldots,(d+1)m\}$. In differential notation we thus have
\[
	\diff \hat{Y}^n(t) = \Psi^n(t) \hat{Y}^n(t) \diff t + n^{-1/2} \diff M^n(t),
\]
from which we derive, with $\zeta^n(t) :=  \int_0^t \Psi^n(s) \diff s $ and $\hat{Z}^n(t) := \exp(-\zeta^n(t)) \hat{Y}^n(t)$,
\begin{equation}\label{Eqn: dZ, TT diffusion}
	\diff \hat{Z}^n(t) = n^{-1/2} \exp(-\zeta^n(t)) \diff M^n(t).
\end{equation}
Our derivation consists of three steps. 
\begin{itemize}
\item[$\circ$]
In the first step we wish to show that $\exp(-\zeta^n(\cdot))$ converges almost surely, u.o.c.\ to $\int_0^\cdot\partial h(\bar{Y}(s)) \diff s$. By \autoref{Prop: TT fluid limit}, $Y^n(t) \overset{\rm a.s.}{\to} \bar{Y}(t)$ u.o.c. As the function $h(\cdot)$ is almost everywhere differentiable, we can rewrite the integral in the definition of $\zeta^n(\cdot)$ up to the null set where the derivative does not exist. Combining these arguments, and using that $\partial h$ is uniformly bounded, an application of the  dominated convergence yields, for $t \geq 0$,
\begin{equation}\label{eq: proof diffusion A, convergence zeta}
	\zeta^n(t) \overset{\rm a.s.}{\to} \zeta(t) := \int_0^t \partial h(\bar{Y}(s)) \diff s.
\end{equation} 
By `continuous mapping' and \cite[Thm.\ VI.2.15]{js2003} (bearing in mind that $x \mapsto -\exp(-x)$ is increasing and continuous), we have 
\begin{equation*}
\sup_{s \leq t} | \exp(-\zeta^n(t)) - {\exp(-\zeta(t))} | \overset{\rm a.s.}{\to} 0.
\end{equation*}

\item[$\circ$]In the second step we show that the martingales $n^{-1/2} M^n(\cdot)$ converge weakly to a $(d+1)m$-dimensional scaled Brownian motion. We want to apply the martingale functional central limit theorem (MFCLT) \cite[Thm.\ VIII.3.22]{js2003}. To this end, we verify the corresponding conditions. For every $n$, $M^n(\cdot)$ is a compensated time-inhomogeneous Poisson process with uniformly bounded rate, implying that $n^{-1/2} M^n(\cdot)$ is locally square integrable, which is the first condition. Moreover, it is immediate that the jumps $n^{-1/2} (M^n(t)-\lim_{s \uparrow t} M^n(s)) \leq n^{-1/2} \to 0$, which by \cite[Eqn.\ (VIII.3.5)]{js2003} is sufficient for the second condition.

Finally we need to verify that the normalized predictable quadratic co-variation matrix, {i.e.,}  $\langle n^{-1/2} M^n \rangle_t$, converges in probability to a deterministic limit. We may assume, without loss of generality, that $\langle M^n_k , M^n_l \rangle_t = 0$, since we can  construct the processes $Y(\cdot)$ as independent unit-rate Poisson processes on a product space, and apply an absolutely continuous change of measure so that $Y(\cdot)$ has the required intensity (cf.\ \cite[Section VI.2]{bremaud1981}). By construction and an application of Girsanov's theorem, we have $\langle M^n_k , M^n_l \rangle_t = 0$ almost surely. 
To conclude our argument, we have
\begin{align*}
	\langle n^{-1/2} M^n \rangle_t & = n^{-1} \langle M^n \rangle_t  
	 = \textrm{diag} \left\{n^{-1} \int_0^{nt} h^n(Y(s)) \diff s \right\} \\
	& = \textrm{diag} \left\{ \int_0^{t} h(Y^n(s)) \diff s \right\} 
 \overset{\rm a.s.}{\to} \textrm{diag} \left\{ \int_0^t h(\bar{Y}(s)) \diff s \right\},
\end{align*} where the convergence is a consequence of \autoref{Prop: TT fluid limit}. By the MFCLT, it now follows that $n^{-1/2} M^n(\cdot)$ converges weakly to a scaled Brownian motion $\tilde{B}(\cdot)$ with $\langle \tilde{B} \rangle_t =\textrm{diag} \{ \int_0^t h(\bar{Y}(s)) \diff s\}$.

\item[$\circ$]For the third and final step, we combine previous results to conclude weak convergence of the stochastic integrals in Equation (\ref{Eqn: dZ, TT diffusion}). We want to use \cite[Thm. VI.6.22]{js2003} for which we need the so-called P-UT property, which follows from our second step and \cite[Prop.\ 6.13]{js2003}. We use Slutsky's lemma to obtain joint weak convergence of $\exp(-\zeta^n(\cdot))$ and $n^{-1/2}M^n(\cdot)$ and conclude that $Z^n(\cdot) = \int_0^{\cdot} n^{-1/2} \exp(-\zeta^n(u))\, {\rm d}M^n(u)$ converges weakly to $\int_0^{\cdot} \exp(-\zeta(u))\, {\rm d}\tilde{B}(u)$. Finally, by applying `Slutsky' again in combination with `continuous mapping', we obtain the weak convergence of $\hat{Y}^n(\cdot)$ to a process $\hat{Y}(\cdot)$ defined through
\begin{equation}\label{YY}
	\diff \hat{Y}(t) = \partial h(\bar{Y}(t)) \hat{Y}(t) \diff t + \diff \tilde{B}(t), \quad \hat{Y}(0) = 0.
\end{equation}\end{itemize}
The stated result follows directly from (\ref{YY}), using the definition of $\tilde B(\cdot).$
\end{proof}

We can now approximate the travel-time distribution by combining  (\ref{Eqn: relation Y and T}) with \autoref{Prop: TT fluid limit} and \autoref{Ydiff}. In the next {section,} this procedure {will be} illustrated through a series of examples.

\section{Numerical Examples} 
\label{NE}

Now that we have established fluid and diffusion limits for both $\rho(\cdot)$ and $Y(\cdot)$, we illustrate their usefulness by a series of numerical experiments. These reproduce traffic phenomena like \textit{forward propagation}, \textit{backward moving jams}, and \textit{shockwave formation}. We start the section, however, by assessing the validity and accuracy of the approximation.

\subsection{Generating Diffusion-based Approximations}
\label{sec: generating diff approx}

In our limit theorems, we have scaled both the length of the cells and time by a factor~$n$. However, this parameter $n$ is `artificial', in that in a real-world situation, such a scaling parameter obviously does not have a physical meaning. In this subsection we point out how our limiting results (for $n$ large) can be converted into approximations (in which $n$ should not appear). 
We then focus on  accuracy of the resulting approximation by comparing it with simulation-based estimates. 

\subsubsection{Approximation; Role of Scaling Parameter $n$} 
We proceed by providing an approximation for the distribution of $\rho(\cdot)$, based on the diffusion limit that we established in Section \ref{SL}. 
By an explicit calculation we then show that this approximation  is independent of the scaling parameter $n$, as it should. 

In \autoref{Thm: Fluid limit cells} and \autoref{Thm: diffusion cells} we state limiting results for the processes $\rho^n(nt)$ and  $\hat{\rho}^n(t):=\,\sqrt{n} (\rho^n(nt) - \bar{\rho}(t))$. These show that
\[
	\rho^n(n\,\cdot) \overset{\rm a.s.}{\to} \bar{\rho}(\cdot)\quad \textrm{and} \quad \hat{\rho}^n(\cdot) \Rightarrow \hat{\rho}(\cdot).
\]
Rewriting gives, for $n$ large, an expression for $\rho^n(\cdot)$ in terms of a fluid limit $\rhob(\cdot)$ and a (zero-mean) Gaussian perturbation around it:
\[
\rho^n(nt) \approx \rhob(t) + \frac{1}{\sqrt{n}} \rhoh(t).
\]
To obtain the prelimit process $\rho(\cdot)$ back on the left-hand side, we reverse the scaling, by slowing time down by a factor $n$ and dividing cell lengths by $\ell_i / n$. This gives the distributional approximation
\begin{equation}\label{Eqn: scaling invariant n rho}
	\rho(t) \approx \bar{\rho}(t/n) + \frac{1}{\sqrt{n}} \hat{\rho}(t/n);
\end{equation}
here $\rho(\cdot)$ corresponds to cells of  length $\ell_i$, whereas the processes on the right-hand side correspond to cells of  length $\ell_i / n$. 
Now we can show that the right-hand side of \eqref{Eqn: scaling invariant n rho} actually does not depend on $n$. For $i\in\{1,\ldots, d\}$ and $j\in\{1,\ldots, m\}$, we have
\begin{align*}
	\bar{\rho}_{ij} (t/n) 
	& = \bar{\rho}_{ij}(0) + \frac{n}{\ell_i} \int_0^{t/n} q_{i-1,j}(\bar{\rho}(s)) \diff s - \frac{n}{\ell_i} \int_0^{t/n} q_{i,j}(\bar{\rho}(s)) \diff s \\ 
	& = \bar{\rho}_{ij}(0) + \frac{n}{\ell_i} \frac{1}{n} \int_0^{t} q_{i-1,j}(\bar{\rho}(\tfrac{s}{n})) \diff s - \frac{n}{\ell_i} \frac{1}{n} \int_0^{t} q_{i,j}(\bar{\rho}(\tfrac{s}{n})) \diff s.
\end{align*}
Now, setting $\bar{\rho}^\circ(t) := \bar{\rho}(\tfrac{t}{n})$, we get
\[
	\bar{\rho}^\circ(t) = \bar{\rho}^\circ(0) + \frac{1}{\ell_i} \int_0^{t} q_{i-1,j}(\bar{\rho}^\circ(s)) \diff s - \frac{1}{\ell_i} \int_0^{t} q_{i,j}(\bar{\rho}^\circ(s)) \diff s.
\]
For $\hat{\rho}(\cdot)$ the calculation is similar. For $i\in\{1,\ldots, d\}$ and $j\in\{1,\ldots, m\}$,
\begin{align*}
 \frac{1}{\sqrt{n}} \rhoh_{ij}(t/n) 
&=  \frac{1}{\sqrt{n}} \frac{n}{\ell_i} \int_0^{t/n} \nabla q_{i-1,j}(\rhob(s)) \cdot \rhoh(s) \diff s 
- \frac{1}{\sqrt{n}} \frac{n}{\ell_i} \int_0^{t/n} \nabla q_{i,j}(\rhob(s)) \cdot \rhoh(s) \diff s \\
& \:\:\:\:+ \frac{1}{\sqrt{n}} \frac{n}{\ell_i} \int_0^{t/n} \sqrt{q_{i-1,j}(\rhob(s))} \diff B_{i-1,j}(s) 
- \frac{1}{\sqrt{n}} \frac{n}{\ell_i} \int_0^{t/n} \sqrt{q_{i,j}(\rhob(s))} \diff B_{i,j}(s) \\
&=  \frac{n}{\ell_i} \frac{1}{n} \int_0^{t} \nabla q_{i-1,j}(\rhob(\tfrac{s}{n})) \cdot \frac{1}{\sqrt{n}} \rhoh(\tfrac{s}{n}) \diff s 
- \frac{n}{\ell_i} \frac{1}{n} \int_0^{t} \nabla q_{i,j}(\rhob(\tfrac{s}{n})) \cdot \frac{1}{\sqrt{n}} \rhoh(\tfrac{s}{n}) \diff s \\
& \:\:\:\:+ \frac{1}{\sqrt{n}} \frac{n}{\ell_i} \frac{1}{\sqrt{n}} \int_0^{t} \sqrt{q_{i-1,j}(\rhob(\tfrac{s}{n}))} \diff B_{i-1,j}(s) 
- \frac{1}{\sqrt{n}} \frac{n}{\ell_i} \frac{1}{\sqrt{n}} \int_0^{t} \sqrt{q_{i,j}(\rhob(\tfrac{s}{n}))} \diff B_{i-1,j}(s),
\end{align*}
where we have used the scale-invariance of Brownian motion, {i.e.,} $B(s/n) \overset{\rm d}{=} B(s) / \sqrt{n}$. By setting $\hat{\rho}^\circ(t) :=  \hat{\rho}(t/n)/\sqrt{n}$, we get
\begin{multline*}
\hat{\rho}^\circ(t) = \frac{1}{\ell_i} \int_0^{t} \nabla q_{i-1,j}(\bar{\rho}^\circ(s)) \cdot \hat{\rho}^\circ(s) \diff s -\frac{1}{\ell_i} \int_0^{t} \nabla q_{i,j}(\bar{\rho}^\circ(s)) \cdot \hat{\rho}^\circ(s) \diff s \\
+ \frac{1}{\ell_i} \int_0^{t} {\sqrt{q_{i-1,j}(\bar{\rho}^\circ(s))} \diff B_{i-1,j}(s) - \frac{1}{\ell_i} \int_0^{t} \sqrt{q_{i,j}(\hat{\rho}^\circ(s))} \diff B_{i,j}(s).}
\end{multline*}
In conclusion, we end up with the distributional approximation
\begin{equation}\label{APPR}
	\rho(\cdot) \approx \bar{\rho}^\circ(\cdot) + \hat{\rho}^\circ(\cdot),
\end{equation}
which is independent of $n$. More specifically, we have that $\rho(\cdot)$ is approximately a Gaussian process, where 
\begin{equation}\label{Eqn: mean, var rho}
 \E [\rho(t)] \approx \bar{\rho}^\circ(t),\:\:\:\textrm{var}[\rho(t)] \approx \textrm{var} [\hat{\rho}^\circ(t)],
\end{equation}
with $\bar{\rho}^\circ(\cdot)$ being deterministic and $\hat{\rho}^\circ(\cdot)$  a zero-mean Gaussian process.

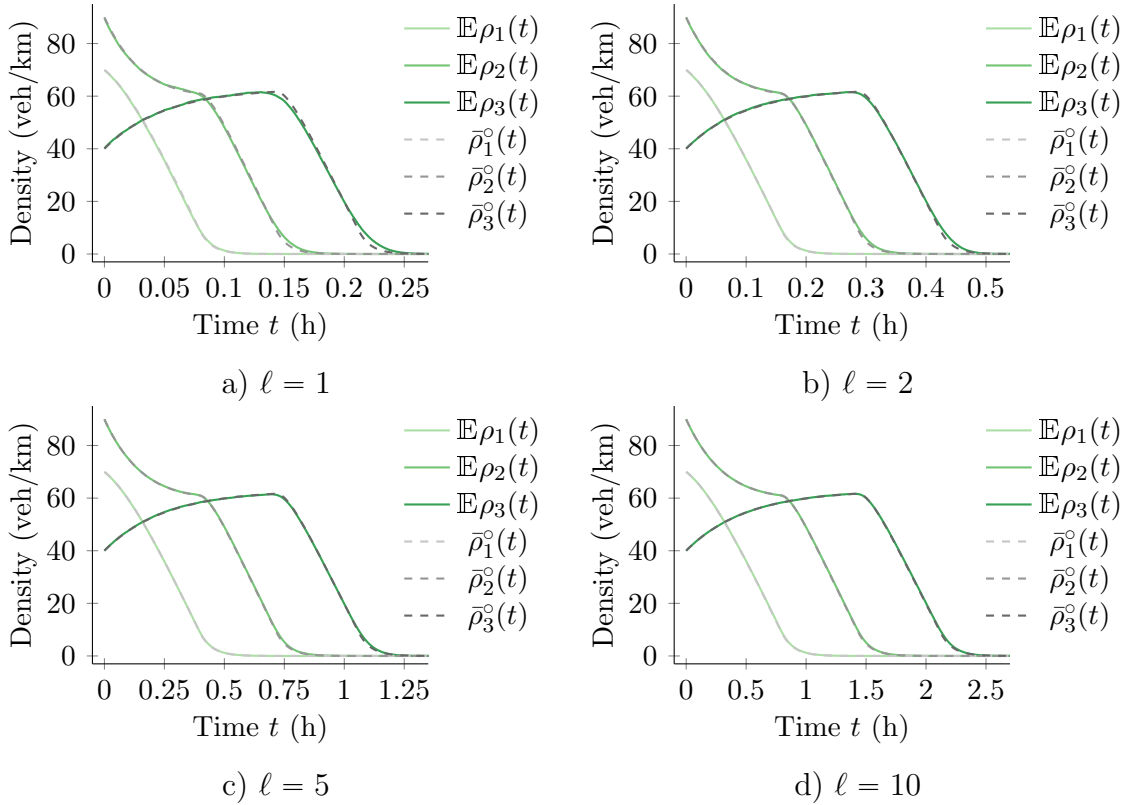
\begin{figure}
\begin{tabular}{cc}

\begin{tikzpicture}	
\begin{axis}[
    xlabel={Time $t$ (h)},    ylabel={Density (veh/km)},
    xmin=-0.01, xmax=0.27,    ymin=-3, ymax=95,
    xtick={0, 0.05, 0.1,0.15,0.2,0.25},    ytick={0,20,40,60,80},
    xticklabel style={/pgf/number format/fixed, /pgf/number format/precision=2},
	scaled x ticks=false,
    xtick pos=left,    ytick pos=left,    axis lines=left,    
    width = 6cm,    height = 5cm,
    no markers,    x axis line style=-,    y axis line style=-,
    y label style={at={(axis description cs:0.15,.5)},anchor=south},
    font = \small,
    legend style={at={(0.9,1)},anchor=north west, draw=none}
]
\addplot+[thick,color11] table [x={time}, y={mean_sim_0}] {pd_mean_df_l=1.dat};
\addplot+[thick,color12] table [x={time}, y={mean_sim_1}] {pd_mean_df_l=1.dat};
\addplot+[thick,color13] table [x={time}, y={mean_sim_2}] {pd_mean_df_l=1.dat};
\addplot+[thick,dashed,color14] table [x={time}, y={mean_0}] {pd_mean_df_l=1.dat};
\addplot+[thick,dashed,color15] table [x={time}, y={mean_1}] {pd_mean_df_l=1.dat};
\addplot+[thick,dashed,color16] table [x={time}, y={mean_2}] {pd_mean_df_l=1.dat};
\legend{$\E\rho_1(t)$,$\E\rho_2(t)$,$\E\rho_3(t)$,
			$\bar{\rho}_1^\circ(t)$,$\bar{\rho}_2^\circ(t)$,$\bar{\rho}_3^\circ(t)$}
\end{axis}
\end{tikzpicture} & 

\begin{tikzpicture}	
\begin{axis}[
    xlabel={Time $t$ (h)},    ylabel={Density (veh/km)},
    xmin=-0.02, xmax=0.54,    ymin=-3, ymax=95,
    xtick={0,0.1,0.2,0.3,0.4,0.5},    ytick={0,20,40,60,80},
    xticklabel style={/pgf/number format/fixed, /pgf/number format/precision=2},
	scaled x ticks=false,
    xtick pos=left,    ytick pos=left,    axis lines=left,
    width = 6cm,    height = 5cm,
    no markers,    x axis line style=-,    y axis line style=-,
    y label style={at={(axis description cs:0.15,.5)},anchor=south},
    font = \small,
    legend style={at={(0.9,1)},anchor=north west, draw=none}
]
\addplot+[thick,color11] table [x={time}, y={mean_sim_0}] {pd_mean_df_l=2.dat};
\addplot+[thick,color12] table [x={time}, y={mean_sim_1}] {pd_mean_df_l=2.dat};
\addplot+[thick,color13] table [x={time}, y={mean_sim_2}] {pd_mean_df_l=2.dat};
\addplot+[thick,dashed,color14] table [x={time}, y={mean_0}] {pd_mean_df_l=2.dat};
\addplot+[thick,dashed,color15] table [x={time}, y={mean_1}] {pd_mean_df_l=2.dat};
\addplot+[thick,dashed,color16] table [x={time}, y={mean_2}] {pd_mean_df_l=2.dat};
\legend{$\E\rho_1(t)$,$\E\rho_2(t)$,$\E\rho_3(t)$,
			$\bar{\rho}_1^\circ(t)$,$\bar{\rho}_2^\circ(t)$,$\bar{\rho}_3^\circ(t)$}
\end{axis}
\end{tikzpicture} \\
a) $ \ell = 1$ & b) $ \ell = 2$ \\ 

\begin{tikzpicture}	
\begin{axis}[
    xlabel={Time $t$ (h)},    ylabel={Density (veh/km)},
    xmin=-0.05, xmax=1.35,    ymin=-3, ymax=95,
    xtick={0,0.25,0.5,0.75,1,1.25},    ytick={0,20,40,60,80},
    xticklabel style={/pgf/number format/fixed, /pgf/number format/precision=2},
	scaled x ticks=false,
    xtick pos=left,    ytick pos=left,    axis lines=left,
    width = 6cm,    height = 5cm,
    no markers,    x axis line style=-,    y axis line style=-,
    y label style={at={(axis description cs:0.15,.5)},anchor=south},
    font = \small,
    legend style={at={(0.9,1)},anchor=north west, draw=none}
]
\addplot+[thick,color11] table [x={time}, y={mean_sim_0}] {pd_mean_df_l=5.dat};
\addplot+[thick,color12] table [x={time}, y={mean_sim_1}] {pd_mean_df_l=5.dat};
\addplot+[thick,color13] table [x={time}, y={mean_sim_2}] {pd_mean_df_l=5.dat};
\addplot+[thick,dashed,color14] table [x={time}, y={mean_0}] {pd_mean_df_l=5.dat};
\addplot+[thick,dashed,color15] table [x={time}, y={mean_1}] {pd_mean_df_l=5.dat};
\addplot+[thick,dashed,color16] table [x={time}, y={mean_2}] {pd_mean_df_l=5.dat};
\legend{$\E\rho_1(t)$,$\E\rho_2(t)$,$\E\rho_3(t)$,
			$\bar{\rho}_1^\circ(t)$,$\bar{\rho}_2^\circ(t)$,$\bar{\rho}_3^\circ(t)$}
\end{axis}
\end{tikzpicture} & 

\begin{tikzpicture}	
\begin{axis}[
    xlabel={Time $t$ (h)},    ylabel={Density (veh/km)},
    xmin=-0.1, xmax=2.7,    ymin=-3, ymax=95,
    xtick={0,0.5,1,1.5,2,2.5},    ytick={0,20,40,60,80},
    xticklabel style={/pgf/number format/fixed, /pgf/number format/precision=2},
	scaled x ticks=false,
    xtick pos=left,    ytick pos=left,    axis lines=left,
    width = 6cm,    height = 5cm,
    no markers,    x axis line style=-,    y axis line style=-,
    y label style={at={(axis description cs:0.15,.5)},anchor=south},
    font = \small,
    legend style={at={(0.9,1)},anchor=north west, draw=none}
]
\addplot+[thick,color11] table [x={time}, y={mean_sim_0}] {pd_mean_df_l=10.dat};
\addplot+[thick,color12] table [x={time}, y={mean_sim_1}] {pd_mean_df_l=10.dat};
\addplot+[thick,color13] table [x={time}, y={mean_sim_2}] {pd_mean_df_l=10.dat};
\addplot+[thick,dashed,color14] table [x={time}, y={mean_0}] {pd_mean_df_l=10.dat};
\addplot+[thick,dashed,color15] table [x={time}, y={mean_1}] {pd_mean_df_l=10.dat};
\addplot+[thick,dashed,color16] table [x={time}, y={mean_2}] {pd_mean_df_l=10.dat};
\legend{$\E\rho_1(t)$,$\E\rho_2(t)$,$\E\rho_3(t)$,
			$\bar{\rho}_1^\circ(t)$,$\bar{\rho}_2^\circ(t)$,$\bar{\rho}_3^\circ(t)$}
\end{axis}
\end{tikzpicture} \\
c) $ \ell = 5$ & d) $ \ell = 10$
\end{tabular}
\caption{Comparison of $\rhob_i^\circ(t)$ and $\E \rho_i(t)$, the latter estimated using simulation, for $i = 1,2,3$.
\label{Fig: comparison mean rho}}
\end{figure}

\subsubsection{Validity and Accuracy of the Approximation}

We proceed by validating the approximation (\ref{APPR}) by simulation. We use the Daganzo MFD, cf.\ Example~\ref{Example: Daganzo} in \autoref{sec: example MFDs}, with parameters $v^f = 100$ km/h, $w = 20$ km/h, $q^{{\rm max}} = 1800$ veh/h, $\rho^{{\rm jam}} = 105$ veh/km, $\lambda = 0$ veh/h and $\nu = 900$ veh/h. We consider a segment with $d = 3$ cells, each of them having length $\ell_i=\ell$, where $\ell \in \{1,2,5,10\}$ km, to illustrate that the accuracy of the approximation improves when  $\ell$ grows. 

For the simulation, we consider a scenario that is initialized with $\rho(0) = (70,90,40)^\top$ and runs on the time-interval $[0,1000\,\ell]$,  in seconds. We estimate the mean and standard deviation of $\rho(k\Delta t)$, for $\Delta t = \ell$ and $k = 0,1,\ldots,1000$, using $1000$ simulated samples. We compare this with the approximated mean and standard deviation. We rely on  \eqref{Eqn: mean, var rho}, where we use \eqref{eqn: dVar rho} to evaluate the variance.

The comparison of the means is given in \autoref{Fig: comparison mean rho}, whereas the comparison for the standard deviations can be found  in \autoref{Fig: comparison var rho}. Both figures show that the difference between the curves becomes smaller when $\ell$ increases. 



\begin{figure}
\begin{tabular}{cc}

\begin{tikzpicture}	
\begin{axis}[
    xlabel={Time $t$ (h)},    ylabel={Density (veh/km)},
    xmin=-0.01, xmax=0.27,    ymin=-0.1, ymax=15,
    xtick={0, 0.05, 0.1,0.15,0.2,0.25},    ytick={0,3,6,9,12,15},
    xticklabel style={/pgf/number format/fixed, /pgf/number format/precision=2},
	scaled x ticks=false,
    xtick pos=left,    ytick pos=left,    axis lines=left,    
    width = 6cm,    height = 5cm,
    no markers,    x axis line style=-,    y axis line style=-,
    y label style={at={(axis description cs:0.15,.5)},anchor=south},
    font = \small,
    legend style={at={(1,1)},anchor=north west, draw=none}
]
\addplot+[thick,color11] table [x={time}, y={std_sim_0}] {pd_std_df_l=1.dat};
\addplot+[thick,color12] table [x={time}, y={std_sim_1}] {pd_std_df_l=1.dat};
\addplot+[thick,color13] table [x={time}, y={std_sim_2}] {pd_std_df_l=1.dat};
\addplot+[thick,dashed,color14] table [x={time}, y={sigma_0}] {pd_std_df_l=1.dat};
\addplot+[thick,dashed,color15] table [x={time}, y={sigma_1}] {pd_std_df_l=1.dat};
\addplot+[thick,dashed,color16] table [x={time}, y={sigma_2}] {pd_std_df_l=1.dat};
\legend{$\sigma(\hat{\rho}^\circ_1(t))$,$\sigma(\hat{\rho}^\circ_2(t))$,$\sigma(\hat{\rho}^\circ_3(t))$,
			 $\sigma(\rho_1(t))$,$\sigma(\rho_2(t))$,$\sigma(\rho_3(t))$}
\end{axis}
\end{tikzpicture} & 

\begin{tikzpicture}	
\begin{axis}[
    xlabel={Time $t$ (h)},    ylabel={Density (veh/km)},
    xmin=-0.02, xmax=0.54,    ymin=-0.05, ymax=10,
    xtick={0,0.1,0.2,0.3,0.4,0.5},    ytick={0,2,4,6,8,10},
    xticklabel style={/pgf/number format/fixed, /pgf/number format/precision=2},
	scaled x ticks=false,
    xtick pos=left,    ytick pos=left,    axis lines=left,
    width = 6cm,    height = 5cm,
    no markers,    x axis line style=-,    y axis line style=-,
    y label style={at={(axis description cs:0.15,.5)},anchor=south},
    font = \small,
    legend style={at={(1,1)},anchor=north west, draw=none}
]
\addplot+[thick,color11] table [x={time}, y={std_sim_0}] {pd_std_df_l=2.dat};
\addplot+[thick,color12] table [x={time}, y={std_sim_1}] {pd_std_df_l=2.dat};
\addplot+[thick,color13] table [x={time}, y={std_sim_2}] {pd_std_df_l=2.dat};
\addplot+[thick,dashed,color14] table [x={time}, y={sigma_0}] {pd_std_df_l=2.dat};
\addplot+[thick,dashed,color15] table [x={time}, y={sigma_1}] {pd_std_df_l=2.dat};
\addplot+[thick,dashed,color16] table [x={time}, y={sigma_2}] {pd_std_df_l=2.dat};
\legend{$\sigma(\hat{\rho}^\circ_1(t))$,$\sigma(\hat{\rho}^\circ_2(t))$,$\sigma(\hat{\rho}^\circ_3(t))$,
			 $\sigma(\rho_1(t))$,$\sigma(\rho_2(t))$,$\sigma(\rho_3(t))$}
\end{axis}
\end{tikzpicture} \\
a) $ \ell = 1$ & b) $ \ell = 2$ \\ 

\begin{tikzpicture}	
\begin{axis}[
    xlabel={Time $t$ (h)},    ylabel={Density (veh/km)},
    xmin=-0.05, xmax=1.35,    ymin=-0.02, ymax=6.2,
    xtick={0,0.25,0.5,0.75,1,1.25},    ytick={0,1,2,3,4,5,6},
    xticklabel style={/pgf/number format/fixed, /pgf/number format/precision=2},
	scaled x ticks=false,
    xtick pos=left,    ytick pos=left,    axis lines=left,
    width = 6cm,    height = 5cm,
    no markers,    x axis line style=-,    y axis line style=-,
    y label style={at={(axis description cs:0.15,.5)},anchor=south},
    font = \small,
    legend style={at={(1,1)},anchor=north west, draw=none}
]
\addplot+[thick,color11] table [x={time}, y={std_sim_0}] {pd_std_df_l=5.dat};
\addplot+[thick,color12] table [x={time}, y={std_sim_1}] {pd_std_df_l=5.dat};
\addplot+[thick,color13] table [x={time}, y={std_sim_2}] {pd_std_df_l=5.dat};
\addplot+[thick,dashed,color14] table [x={time}, y={sigma_0}] {pd_std_df_l=5.dat};
\addplot+[thick,dashed,color15] table [x={time}, y={sigma_1}] {pd_std_df_l=5.dat};
\addplot+[thick,dashed,color16] table [x={time}, y={sigma_2}] {pd_std_df_l=5.dat};
\legend{$\sigma(\hat{\rho}^\circ_1(t))$,$\sigma(\hat{\rho}^\circ_2(t))$,$\sigma(\hat{\rho}^\circ_3(t))$,
			 $\sigma(\rho_1(t))$,$\sigma(\rho_2(t))$,$\sigma(\rho_3(t))$}
\end{axis}
\end{tikzpicture} & 

\begin{tikzpicture}	
\begin{axis}[
    xlabel={Time $t$ (h)},    ylabel={Density (veh/km)},
    xmin=-0.1, xmax=2.7,    ymin=-0.01, ymax=4.5,
    xtick={0,0.5,1,1.5,2,2.5},    ytick={0,1,2,3,4},
    xticklabel style={/pgf/number format/fixed, /pgf/number format/precision=2},
	scaled x ticks=false,
    xtick pos=left,    ytick pos=left,    axis lines=left,
    width = 6cm,    height = 5cm,
    no markers,    x axis line style=-,    y axis line style=-,
    y label style={at={(axis description cs:0.15,.5)},anchor=south},
    font = \small,
    legend style={at={(1,1)},anchor=north west, draw=none}
]
\addplot+[thick,color11] table [x={time}, y={std_sim_0}] {pd_std_df_l=10.dat};
\addplot+[thick,color12] table [x={time}, y={std_sim_1}] {pd_std_df_l=10.dat};
\addplot+[thick,color13] table [x={time}, y={std_sim_2}] {pd_std_df_l=10.dat};
\addplot+[thick,dashed,color14] table [x={time}, y={sigma_0}] {pd_std_df_l=10.dat};
\addplot+[thick,dashed,color15] table [x={time}, y={sigma_1}] {pd_std_df_l=10.dat};
\addplot+[thick,dashed,color16] table [x={time}, y={sigma_2}] {pd_std_df_l=10.dat};
\legend{$\sigma(\hat{\rho}^\circ_1(t))$,$\sigma(\hat{\rho}^\circ_2(t))$,$\sigma(\hat{\rho}^\circ_3(t))$,
			 $\sigma(\rho_1(t))$,$\sigma(\rho_2(t))$,$\sigma(\rho_3(t))$}
\end{axis}
\end{tikzpicture} \\
c) $ \ell = 5$ & d) $ \ell = 10$
\end{tabular}
\caption{Comparison of $\sigma(\hat{\rho}^\circ_i(t)) := \sqrt{\textrm{var}[\rhoh_i^\circ(t)]}$ and $\sigma(\rho_i(t)) := \sqrt{\textrm{var}[\rho_i(t)]}$, the latter estimated using simulation, for $i = 1,2,3$. 
\label{Fig: comparison var rho}}
\end{figure}
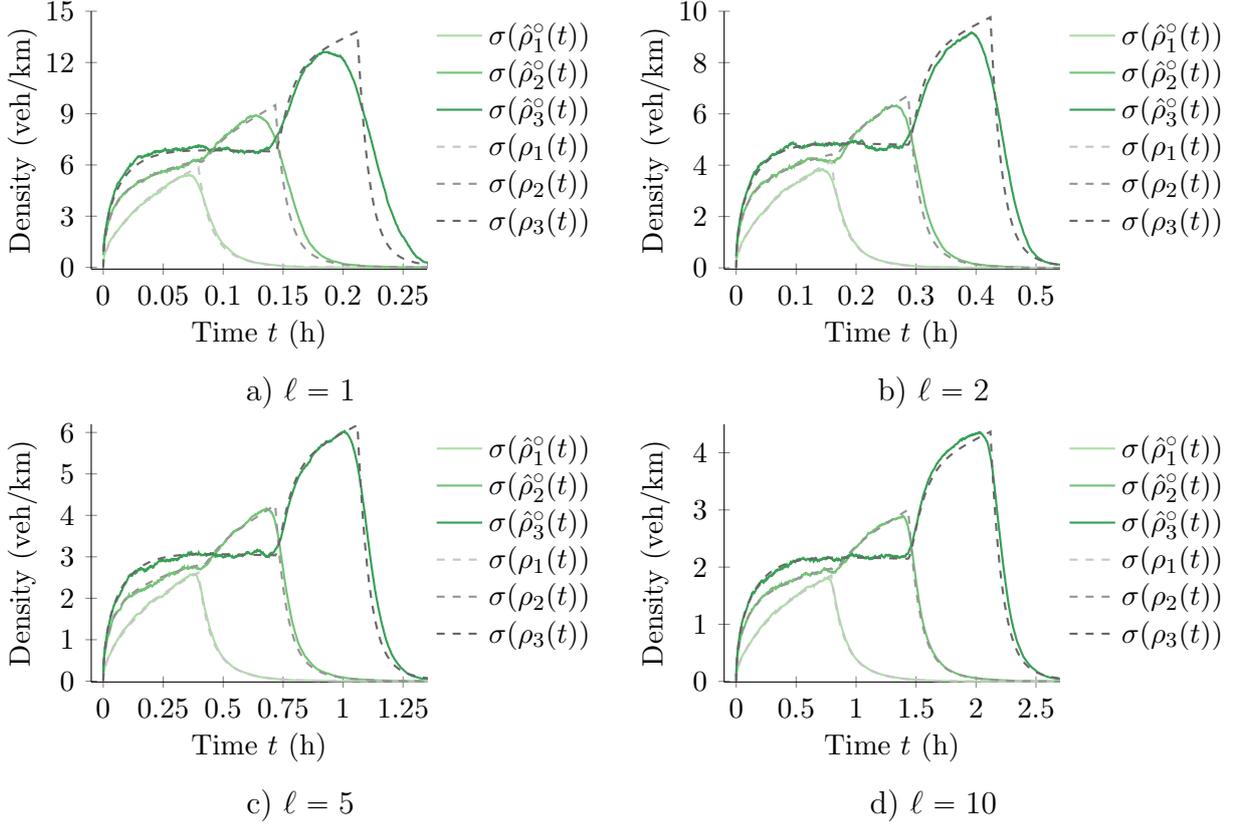

%

\subsection{Numerical Experiments}

In the remainder of this {section,} we illustrate the use of the limit theorems through a series of numerical experiments. Traffic researchers have tried to replicate various phenomena in traffic, and did so quite well with deterministic kinematic wave models, thus recovering the fluid limit. In this section we illustrate the usefulness of our results, in that, in addition to replicating these phenomena with the fluid limit, we are also capable of providing an approximation of the associated vehicle density variance and covariance (where we have covariances in both the spatial and temporal sense). In addition we demonstrate our procedure to approximate {travel-time} distributions.

For our experiments, we take the multi-class MFD from \cite{cb2003} (corresponding with Example~\ref{Example: CB} in \autoref{sec: example MFDs}), with parameters that are used in the numerical experiments of the original paper, with a kilometer taken as unit length. We have $v^f_1 = 108$ km/h, $v^f_2 = 79.2$ km/h, $v^{\rm c} = 61.2$ km/h, $L_1 = 0.0065$ km, $L_2 = 0.0165$ km, $N = 3$ and $\beta = 0.25$. We refer to class~$1$ vehicles as \textit{cars} and to the class~$2$ vehicles as \textit{trucks}.

To obtain the Gaussian approximation of   the mean and variance, we jointly solve the differential form of \eqref{Eqn: fluid limit from, for (ij)} and \eqref{eqn: dVar rho} using a numerical solver. 
To obtain the matrix $\partial Q$, one takes, for each regular point of $Q$, the derivative of the discrete flux-functions. In every non-regular point, one assigns the maximum velocity of a vehicle to the components that are not differentiable, so that $\partial Q$ is bounded by a Lipschitz constant. Additionally, to obtain the travel-time distribution, we use the analog of \eqref{eqn: Cov rho_t, rho_s} for $\hat{Y}$, so that we can evaluate the tail probabilities in \eqref{Eqn: relation Y and T} based on the joint normal approximation of $Y_{i+k,j}(t+x)$ and $Y_{i-1,j}(t)$.


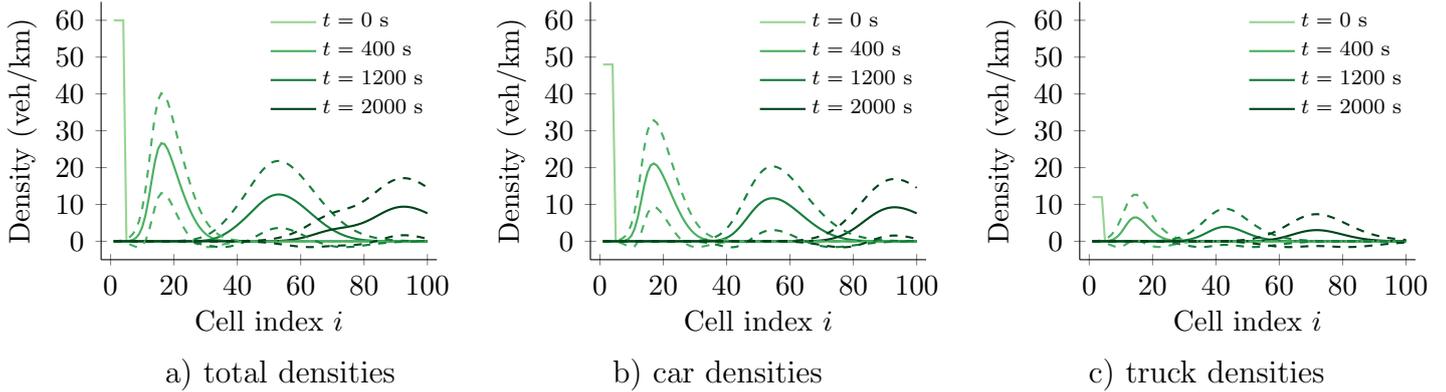
\begin{figure}
\begin{tabular}{ccc}
\hspace{-1.5cm}
\begin{tikzpicture}	
\begin{axis}[
    xlabel={Cell index $i$},    ylabel={Density (veh/km)},
    xmin=-3, xmax=103,    ymin=-5, ymax=65,
    xtick={0,20,40,60,80,100},    ytick={0,10,20,30,40,50,60},
    xticklabel style={/pgf/number format/fixed, /pgf/number format/precision=2},
	scaled x ticks=false,
    xtick pos=left,    ytick pos=left,    axis lines=left,
    width = 6cm,    height = 5cm,
    no markers,    x axis line style=-,    y axis line style=-,
    y label style={at={(axis description cs:0.12,.5)},anchor=south},
    font = \small,
    legend style={at={(1,1)},anchor=north east, draw=none, font=\tiny},
    legend cell align={left}
]
\addplot+[thick,color21] table [x={space}, y={rho_t0}] {pd_CB1_sum_df.dat};
\addplot+[thick,color22] table [x={space}, y={rho_t400}] {pd_CB1_sum_df.dat};
\addplot+[thick,color23] table [x={space}, y={rho_t1200}] {pd_CB1_sum_df.dat};
\addplot+[thick,color24] table [x={space}, y={rho_t2000}] {pd_CB1_sum_df.dat};
\addplot+[thick,dashed,color22] table [x={space}, y={rho_t400_varup}] {pd_CB1_sum_df.dat};
\addplot+[thick,dashed,color22] table [x={space}, y={rho_t400_vardown}
]{pd_CB1_sum_df.dat};
\addplot+[thick,dashed,color23] table [x={space}, y={rho_t1200_varup}] {pd_CB1_sum_df.dat};
\addplot+[thick,dashed,color23] table [x={space}, y={rho_t1200_vardown}
]{pd_CB1_sum_df.dat};
\addplot+[thick,dashed,color24] table [x={space}, y={rho_t2000_varup}] {pd_CB1_sum_df.dat};
\addplot+[thick,dashed,color24] table [x={space}, y={rho_t2000_vardown}
]{pd_CB1_sum_df.dat};
\legend{$t=0$ s, $t = 400$ s, $t = 1200$ s, $t=2000$ s}
\end{axis}
\end{tikzpicture} &

\begin{tikzpicture}	
\begin{axis}[
    xlabel={Cell index $i$},    ylabel={Density (veh/km)},
    xmin=-3, xmax=103,    ymin=-5, ymax=65,
    xtick={0,20,40,60,80,100},    ytick={0,10,20,30,40,50,60},
    xticklabel style={/pgf/number format/fixed, /pgf/number format/precision=2},
	scaled x ticks=false,
    xtick pos=left,    ytick pos=left,    axis lines=left,
    width = 6cm,    height = 5cm,
    no markers,    x axis line style=-,    y axis line style=-,
    y label style={at={(axis description cs:0.12,.5)},anchor=south},
    font = \small,
    legend style={at={(1,1)},anchor=north east, draw=none, font=\tiny},
    legend cell align={left}
]
\addplot+[thick,color21] table [x={space}, y={rho_t0}] {pd_CB1_type1_df.dat};
\addplot+[thick,color22] table [x={space}, y={rho_t400}] {pd_CB1_type1_df.dat};
\addplot+[thick,color23] table [x={space}, y={rho_t1200}] {pd_CB1_type1_df.dat};
\addplot+[thick,color24] table [x={space}, y={rho_t2000}] {pd_CB1_type1_df.dat};
\addplot+[thick,dashed,color22] table [x={space}, y={rho_t400_varup}] {pd_CB1_type1_df.dat};
\addplot+[thick,dashed,color22] table [x={space}, y={rho_t400_vardown}
]{pd_CB1_type1_df.dat};
\addplot+[thick,dashed,color23] table [x={space}, y={rho_t1200_varup}] {pd_CB1_type1_df.dat};
\addplot+[thick,dashed,color23] table [x={space}, y={rho_t1200_vardown}
]{pd_CB1_type1_df.dat};
\addplot+[thick,dashed,color24] table [x={space}, y={rho_t2000_varup}] {pd_CB1_type1_df.dat};
\addplot+[thick,dashed,color24] table [x={space}, y={rho_t2000_vardown}
]{pd_CB1_type1_df.dat};
\legend{$t=0$ s, $t = 400$ s, $t = 1200$ s, $t=2000$ s}
\end{axis}
\end{tikzpicture} & 

\begin{tikzpicture}	
\begin{axis}[
    xlabel={Cell index $i$},    ylabel={Density (veh/km)},
    xmin=-3, xmax=103,    ymin=-5, ymax=65,
    xtick={0,20,40,60,80,100},    ytick={0,10,20,30,40,50,60},
    xticklabel style={/pgf/number format/fixed, /pgf/number format/precision=2},
	scaled x ticks=false,
    xtick pos=left,    ytick pos=left,    axis lines=left,
    width = 6cm,    height = 5cm,
    no markers,    x axis line style=-,    y axis line style=-,
    y label style={at={(axis description cs:0.12,.5)},anchor=south},
    font = \small,
    legend style={at={(1,1)},anchor=north east, draw=none, font=\tiny},
    legend cell align={left}
]
\addplot+[thick,color21] table [x={space}, y={rho_t0}] {pd_CB1_type2_df.dat};
\addplot+[thick,color22] table [x={space}, y={rho_t400}] {pd_CB1_type2_df.dat};
\addplot+[thick,color23] table [x={space}, y={rho_t1200}] {pd_CB1_type2_df.dat};
\addplot+[thick,color24] table [x={space}, y={rho_t2000}] {pd_CB1_type2_df.dat};
\addplot+[thick,dashed,color22] table [x={space}, y={rho_t400_varup}] {pd_CB1_type2_df.dat};
\addplot+[thick,dashed,color22] table [x={space}, y={rho_t400_vardown}
]{pd_CB1_type2_df.dat};
\addplot+[thick,dashed,color23] table [x={space}, y={rho_t1200_varup}] {pd_CB1_type2_df.dat};
\addplot+[thick,dashed,color23] table [x={space}, y={rho_t1200_vardown}
]{pd_CB1_type2_df.dat};
\addplot+[thick,dashed,color24] table [x={space}, y={rho_t2000_varup}] {pd_CB1_type2_df.dat};
\addplot+[thick,dashed,color24] table [x={space}, y={rho_t2000_vardown}
]{pd_CB1_type2_df.dat};
\legend{$t=0$ s, $t = 400$ s, $t = 1200$ s, $t=2000$ s}
\end{axis}
\end{tikzpicture} \\
a) total densities & b) car densities & c) truck densities
\end{tabular}
\caption{Solution differential equations for $\rhob(t)$ and $\textrm{var}[\rhoh(t)]$, Example \ref{E1} - forward propagation.}
\label{fig: results experiment 1}
\end{figure}

\begin{example}[Forward Propagation]\label{E1}{\em

In the first {experiment, }we use the results of \autoref{Thm: Fluid limit cells} and \autoref{Thm: diffusion cells} to study forward propagation. In the experiment, the first 5 cells initially have a density of 60 veh/km of which 20\% are trucks, and we let the system evolve over a time interval of 2000 s, with $\Delta t = 2$ s. The experiment is taken from \cite[Fig.\ 3]{cb2003}, where $d = 100$, but with $\ell_i = 600$ m for each $i$. Under these parameters, the Godunov scheme is stable (cf. \cite[Section 10.6 and Eqn.\  (13.11)]{leveque1992}).

\autoref{fig: results experiment 1} shows the results of our experiment, with from left to right the total vehicle densities, the car densities, and the truck densities, where the dotted lines show the corresponding $95\%$ confidence intervals. Note that the trucks propagate slower on average than the cars, and that they dissipate slower from the initial density; they are more clumped. }
\end{example}

\begin{example}[Travel Times, for the setting of \autoref{E1}]\label{E2}{\em 

To illustrate our approximation of the travel-time distribution, we consider the time that it takes a vehicle in cell~$10$ at time $t = 200$ s, to reach cell~$50$. To be precise, we evaluate \eqref{Eqn: relation Y and T} with $Y$ replaced by $\hat{Y}$, $t=200$ s, $i=10$, $k=39$ and $x \in \{0,2,4,\ldots,2000\}$, in the setting of the above forward-propagation experiment. 

First, in \autoref{Fig: evolution hat{Y}}, we have plotted the mean evolution of $\hat{Y}$ with its associated $95 \%$ confidence interval, at different points in time. The vertical gray lines are associated with the position of cell~$10$ and cell~$50$, whereas the horizontal gray lines correspond with the peak of  the car and truck density in cell~$10$ at time $t=200$ s. From panel (b), we see that at time $t=200$~s, $\bar{Y}_{10,1} \approx 93$ and $\bar{Y}_{10,2} \approx 17$. From panels (c)--(f) we observe that the cars arrive at cell 50 (roughly) between $t =1050$ s and $t = 1150$ s, and the trucks (roughly) between $t=1300$ s and $t = 1400$~s.

In \autoref{Fig: Travel Time Distribution and Density} we have plotted both the approximated distribution $\hat{F}_{T_j}(\cdot)$ and density $\hat{f}_{T_j}(\cdot)$ of the travel time of {type-$j$} vehicles ($j=1,2$), as derived from $\hat{Y}$ in \autoref{Fig: evolution hat{Y}}. 
We remark that one can empirically show that the travel-time distribution for cars is close to normal. However, regarding the travel-time distribution for trucks a QQ-plot reveals a distribution with significantly lighter tails.
In general, from (\ref{Eqn: relation Y and T}) we observe that the asymptotic normality of the counting processes $Y_{i,j}(\cdot)$ does not imply asymptotic normality of the travel times. 

%

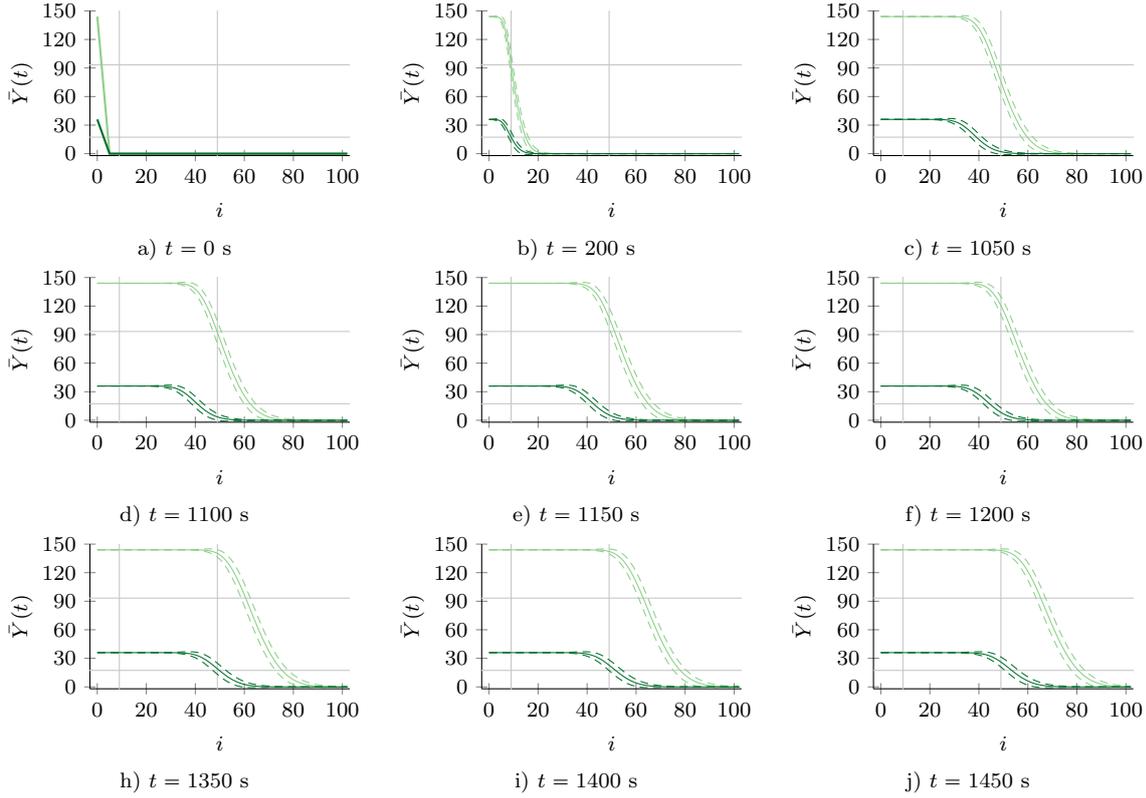
\begin{figure}
\begin{tiny}
\begin{tabular}{ccc}
\begin{tikzpicture}	
\begin{axis}[
    xlabel={$i$},    ylabel={$\bar Y(t)$},
    xmin=-3, xmax=103,    ymin=-2, ymax=150,
    xtick={0,20,40,60,80,100},    ytick={0,30,60,90,120,150},
    xticklabel style={/pgf/number format/fixed, /pgf/number format/precision=2},
	scaled x ticks=false,
    xtick pos=left,    ytick pos=left,    axis lines=left,
    width = 5cm,    height = 3.5cm,
    no markers,    x axis line style=-,    y axis line style=-,
    y label style={at={(axis description cs:0.17,.5)},anchor=south},
    font = \tiny,
]
\addplot[thin, domain=-10:110,mark=none, color14, samples=2] {93.29};
\addplot[thin, domain=-10:110,mark=none, color14, samples=2] {17.32};
\addplot+[thin,solid,mark=none, color14] coordinates {(9, -10) (9, 210)};
\addplot+[thin,solid,mark=none, color14] coordinates {(49, -10) (49, 210)};
\addplot+[thick,solid,color32] table [x={space}, y={Y1_t0}] {pd_Y_df.dat};
\addplot+[thick,solid,color35] table [x={space}, y={Y2_t0}] {pd_Y_df.dat};
\end{axis}
\end{tikzpicture} &
\begin{tikzpicture}	
\begin{axis}[
    xlabel={$i$},    ylabel={$\bar Y(t)$},
    xmin=-3, xmax=103,    ymin=-2, ymax=150,
    xtick={0,20,40,60,80,100},    ytick={0,30,60,90,120,150},
    xticklabel style={/pgf/number format/fixed, /pgf/number format/precision=2},
	scaled x ticks=false,
    xtick pos=left,    ytick pos=left,    axis lines=left,
    width = 5cm,    height = 3.5cm,
    no markers,    x axis line style=-,    y axis line style=-,
    y label style={at={(axis description cs:0.17,.5)}, anchor=south},
    font = \tiny,
]
\addplot[thin, domain=-10:110,mark=none, color14, samples=2] {93.29};
\addplot[thin, domain=-10:110,mark=none, color14, samples=2] {17.32};
\addplot+[thin,solid,mark=none, color14] coordinates {(9, -10) (9, 210)};
\addplot+[thin,solid,mark=none, color14] coordinates {(49, -10) (49, 210)};
\addplot+[solid,color32] table [x={space}, y={Y1_t200}] {pd_Y_df.dat};
\addplot+[densely dashed,color32] table [x={space}, y={Y1_t200_errorup}] {pd_Y_df.dat};
\addplot+[densely dashed,color32] table [x={space}, y={Y1_t200_errordown}] {pd_Y_df.dat};
\addplot+[solid,color35] table [x={space}, y={Y2_t200}] {pd_Y_df.dat};
\addplot+[densely dashed,color35] table [x={space}, y={Y2_t200_errorup}] {pd_Y_df.dat};
\addplot+[densely dashed,color35] table [x={space}, y={Y2_t200_errordown}] {pd_Y_df.dat};
\end{axis}
\end{tikzpicture} & 
\begin{tikzpicture}	
\begin{axis}[
    xlabel={$i$},    ylabel={$\bar Y(t)$},
    xmin=-3, xmax=103,    ymin=-2, ymax=150,
    xtick={0,20,40,60,80,100},    ytick={0,30,60,90,120,150},
    xticklabel style={/pgf/number format/fixed, /pgf/number format/precision=2},
	scaled x ticks=false,
    xtick pos=left,    ytick pos=left,    axis lines=left,
    width = 5cm,    height = 3.5cm,
    no markers,    x axis line style=-,    y axis line style=-,
    y label style={at={(axis description cs:0.17,.5)}, anchor=south},
    font = \tiny,
]
\addplot[thin, domain=-10:110,mark=none, color14, samples=2] {93.29};
\addplot[thin, domain=-10:110,mark=none, color14, samples=2] {17.32};
\addplot+[thin,solid,mark=none, color14] coordinates {(9, -10) (9, 210)};
\addplot+[thin,solid,mark=none, color14] coordinates {(49, -10) (49, 210)};
\addplot+[solid,color32] table [x={space}, y={Y1_t1050}] {pd_Y_df.dat};
\addplot+[densely dashed,color32] table [x={space}, y={Y1_t1050_errorup}] {pd_Y_df.dat};
\addplot+[densely dashed,color32] table [x={space}, y={Y1_t1050_errordown}] {pd_Y_df.dat};
\addplot+[solid,color35] table [x={space}, y={Y2_t1050}] {pd_Y_df.dat};
\addplot+[densely dashed,color35] table [x={space}, y={Y2_t1050_errorup}] {pd_Y_df.dat};
\addplot+[densely dashed,color35] table [x={space}, y={Y2_t1050_errordown}] {pd_Y_df.dat};
\end{axis}
\end{tikzpicture} \\
a) $t = 0$ s & b) $t = 200$ s & c) $t = 1050$ s \\
\begin{tikzpicture}	
\begin{axis}[
    xlabel={$i$},    ylabel={$\bar Y(t)$},
    xmin=-3, xmax=103,    ymin=-2, ymax=150,
    xtick={0,20,40,60,80,100},    ytick={0,30,60,90,120,150},
    xticklabel style={/pgf/number format/fixed, /pgf/number format/precision=2},
	scaled x ticks=false,
    xtick pos=left,    ytick pos=left,    axis lines=left,
    width = 5cm,    height = 3.5cm,
    no markers,    x axis line style=-,    y axis line style=-,
    y label style={at={(axis description cs:0.17,.5)}, anchor=south},
    font = \tiny,
]
\addplot[thin, domain=-10:110,mark=none, color14, samples=2] {93.29};
\addplot[thin, domain=-10:110,mark=none, color14, samples=2] {17.32};
\addplot+[thin,solid,mark=none, color14] coordinates {(9, -10) (9, 210)};
\addplot+[thin,solid,mark=none, color14] coordinates {(49, -10) (49, 210)};
\addplot+[solid,color32] table [x={space}, y={Y1_t1100}] {pd_Y_df.dat};
\addplot+[densely dashed,color32] table [x={space}, y={Y1_t1100_errorup}] {pd_Y_df.dat};
\addplot+[densely dashed,color32] table [x={space}, y={Y1_t1100_errordown}] {pd_Y_df.dat};
\addplot+[solid,color35] table [x={space}, y={Y2_t1100}] {pd_Y_df.dat};
\addplot+[densely dashed,color35] table [x={space}, y={Y2_t1100_errorup}] {pd_Y_df.dat};
\addplot+[densely dashed,color35] table [x={space}, y={Y2_t1100_errordown}] {pd_Y_df.dat};
\end{axis}
\end{tikzpicture} &
\begin{tikzpicture}	
\begin{axis}[
    xlabel={$i$},    ylabel={$\bar Y(t)$},
    xmin=-3, xmax=103,    ymin=-2, ymax=150,
    xtick={0,20,40,60,80,100},    ytick={0,30,60,90,120,150},
    xticklabel style={/pgf/number format/fixed, /pgf/number format/precision=2},
	scaled x ticks=false,
    xtick pos=left,    ytick pos=left,    axis lines=left,
    width = 5cm,    height = 3.5cm,
    no markers,    x axis line style=-,    y axis line style=-,
    y label style={at={(axis description cs:0.17,.5)}, anchor=south},
    font = \tiny,
]
\addplot[thin, domain=-10:110,mark=none, color14, samples=2] {93.29};
\addplot[thin, domain=-10:110,mark=none, color14, samples=2] {17.32};
\addplot+[thin,solid,mark=none, color14] coordinates {(9, -10) (9, 210)};
\addplot+[thin,solid,mark=none, color14] coordinates {(49, -10) (49, 210)};
\addplot+[solid,color32] table [x={space}, y={Y1_t1150}] {pd_Y_df.dat};
\addplot+[densely dashed,color32] table [x={space}, y={Y1_t1150_errorup}] {pd_Y_df.dat};
\addplot+[densely dashed,color32] table [x={space}, y={Y1_t1150_errordown}] {pd_Y_df.dat};
\addplot+[solid,color35] table [x={space}, y={Y2_t1150}] {pd_Y_df.dat};
\addplot+[densely dashed,color35] table [x={space}, y={Y2_t1150_errorup}] {pd_Y_df.dat};
\addplot+[densely dashed,color35] table [x={space}, y={Y2_t1150_errordown}] {pd_Y_df.dat};
\end{axis}
\end{tikzpicture} &
\begin{tikzpicture}	
\begin{axis}[
    xlabel={$i$},    ylabel={$\bar Y(t)$},
    xmin=-3, xmax=103,    ymin=-2, ymax=150,
    xtick={0,20,40,60,80,100},    ytick={0,30,60,90,120,150},
    xticklabel style={/pgf/number format/fixed, /pgf/number format/precision=2},
	scaled x ticks=false,
    xtick pos=left,    ytick pos=left,    axis lines=left,
    width = 5cm,    height = 3.5cm,
    no markers,    x axis line style=-,    y axis line style=-,
    y label style={at={(axis description cs:0.17,.5)}, anchor=south},
    font = \tiny,
]
\addplot[thin, domain=-10:110,mark=none, color14, samples=2] {93.29};
\addplot[thin, domain=-10:110,mark=none, color14, samples=2] {17.32};
\addplot+[thin,solid,mark=none, color14] coordinates {(9, -10) (9, 210)};
\addplot+[thin,solid,mark=none, color14] coordinates {(49, -10) (49, 210)};
\addplot+[solid,color32] table [x={space}, y={Y1_t1200}] {pd_Y_df.dat};
\addplot+[densely dashed,color32] table [x={space}, y={Y1_t1200_errorup}] {pd_Y_df.dat};
\addplot+[densely dashed,color32] table [x={space}, y={Y1_t1200_errordown}] {pd_Y_df.dat};
\addplot+[solid,color35] table [x={space}, y={Y2_t1200}] {pd_Y_df.dat};
\addplot+[densely dashed,color35] table [x={space}, y={Y2_t1200_errorup}] {pd_Y_df.dat};
\addplot+[densely dashed,color35] table [x={space}, y={Y2_t1200_errordown}] {pd_Y_df.dat};
\end{axis}
\end{tikzpicture} \\
d) $t = 1100$ s & e) $t = 1150$ s & f) $t = 1200$ s \\
\begin{tikzpicture}	
\begin{axis}[
    xlabel={$i$},    ylabel={$\bar Y(t)$},
    xmin=-3, xmax=103,    ymin=-2, ymax=150,
    xtick={0,20,40,60,80,100},    ytick={0,30,60,90,120,150},
    xticklabel style={/pgf/number format/fixed, /pgf/number format/precision=2},
	scaled x ticks=false,
    xtick pos=left,    ytick pos=left,    axis lines=left,
    width = 5cm,    height = 3.5cm,
    no markers,    x axis line style=-,    y axis line style=-,
    y label style={at={(axis description cs:0.17,.5)}, anchor=south},
    font = \tiny,
]
\addplot[thin, domain=-10:110,mark=none, color14, samples=2] {93.29};
\addplot[thin, domain=-10:110,mark=none, color14, samples=2] {17.32};
\addplot+[thin,solid,mark=none, color14] coordinates {(9, -10) (9, 210)};
\addplot+[thin,solid,mark=none, color14] coordinates {(49, -10) (49, 210)};
\addplot+[solid,color32] table [x={space}, y={Y1_t1350}] {pd_Y_df.dat};
\addplot+[densely dashed,color32] table [x={space}, y={Y1_t1350_errorup}] {pd_Y_df.dat};
\addplot+[densely dashed,color32] table [x={space}, y={Y1_t1350_errordown}] {pd_Y_df.dat};
\addplot+[solid,color35] table [x={space}, y={Y2_t1350}] {pd_Y_df.dat};
\addplot+[densely dashed,color35] table [x={space}, y={Y2_t1350_errorup}] {pd_Y_df.dat};
\addplot+[densely dashed,color35] table [x={space}, y={Y2_t1350_errordown}] {pd_Y_df.dat};
\end{axis}
\end{tikzpicture} &
\begin{tikzpicture}	
\begin{axis}[
    xlabel={$i$},    ylabel={$\bar Y(t)$},
    xmin=-3, xmax=103,    ymin=-2, ymax=150,
    xtick={0,20,40,60,80,100},    ytick={0,30,60,90,120,150},
    xticklabel style={/pgf/number format/fixed, /pgf/number format/precision=2},
	scaled x ticks=false,
    xtick pos=left,    ytick pos=left,    axis lines=left,
    width = 5cm,    height = 3.5cm,
    no markers,    x axis line style=-,    y axis line style=-,
    y label style={at={(axis description cs:0.17,.5)}, anchor=south},
    font = \tiny,
]
\addplot[thin, domain=-10:110,mark=none, color14, samples=2] {93.29};
\addplot[thin, domain=-10:110,mark=none, color14, samples=2] {17.32};
\addplot+[thin,solid,mark=none, color14] coordinates {(9, -10) (9, 210)};
\addplot+[thin,solid,mark=none, color14] coordinates {(49, -10) (49, 210)};
\addplot+[solid,color32] table [x={space}, y={Y1_t1400}] {pd_Y_df.dat};
\addplot+[densely dashed,color32] table [x={space}, y={Y1_t1400_errorup}] {pd_Y_df.dat};
\addplot+[densely dashed,color32] table [x={space}, y={Y1_t1400_errordown}] {pd_Y_df.dat};
\addplot+[solid,color35] table [x={space}, y={Y2_t1400}] {pd_Y_df.dat};
\addplot+[densely dashed,color35] table [x={space}, y={Y2_t1400_errorup}] {pd_Y_df.dat};
\addplot+[densely dashed,color35] table [x={space}, y={Y2_t1400_errordown}] {pd_Y_df.dat};
\end{axis}
\end{tikzpicture} &
\begin{tikzpicture}	
\begin{axis}[
    xlabel={$i$},    ylabel={$\bar Y(t)$},
    xmin=-3, xmax=103,    ymin=-2, ymax=150,
    xtick={0,20,40,60,80,100},    ytick={0,30,60,90,120,150},
    xticklabel style={/pgf/number format/fixed, /pgf/number format/precision=2},
	scaled x ticks=false,
    xtick pos=left,    ytick pos=left,    axis lines=left,
    width = 5cm,    height = 3.5cm,
    no markers,    x axis line style=-,    y axis line style=-,
    y label style={at={(axis description cs:0.17,.5)}, anchor=south},
    font = \tiny,
]
\addplot[thin, domain=-10:110,mark=none, color14, samples=2] {93.29};
\addplot[thin, domain=-10:110,mark=none, color14, samples=2] {17.32};
\addplot+[thin,solid,mark=none, color14] coordinates {(9, -10) (9, 210)};
\addplot+[thin,solid,mark=none, color14] coordinates {(49, -10) (49, 210)};
\addplot+[solid,color32] table [x={space}, y={Y1_t1450}] {pd_Y_df.dat};
\addplot+[densely dashed,color32] table [x={space}, y={Y1_t1450_errorup}] {pd_Y_df.dat};
\addplot+[densely dashed,color32] table [x={space}, y={Y1_t1450_errordown}] {pd_Y_df.dat};
\addplot+[solid,color35] table [x={space}, y={Y2_t1450}] {pd_Y_df.dat};
\addplot+[densely dashed,color35] table [x={space}, y={Y2_t1450_errorup}] {pd_Y_df.dat};
\addplot+[densely dashed,color35] table [x={space}, y={Y2_t1450_errordown}] {pd_Y_df.dat};
\end{axis}
\end{tikzpicture} \\
h) $t = 1350$ s & i) $t = 1400$ s & j) $t = 1450$ s \\
\end{tabular}
\end{tiny}
\caption{Evolution of $\bar{Y}$, with dashed line marking the 95\% confidence interval, given by $\textrm{var}[\hat{Y}(t)]$; the light green line representing $\bar Y_{i,1}(t)$ and the dark green line $\bar Y_{i,2}(t)$.}
\label{Fig: evolution hat{Y}}
\end{figure}

}\end{example}

\begin{example}[Backward Moving Jam]\label{E3}{\em

This experiment illustrates our results in the case of a backward moving traffic jam. This phenomenon in traffic is seen when there is a large density on a given road segment, with the front dissipating and the back absorbing arriving cars.The jam  thus seems to be moving backward on the road; detailed background on moving jams can be found in{, e.g.,~}\cite{kern}. 

In our {numerical experiments,} we set $d=20$, with $\ell_i = 1$~km for every $i$, with initial density of 300~veh/km in cells 8 up to (and including) 12, and an initial density of 88~veh/km in the other cells. Moreover, we set $\lambda_1 = 4800$~veh/h, $\lambda_2 = 960$~veh/h and $\nu = 1.2$, and we simulate for 1500 seconds with $\Delta t = 1$~s.

\autoref{fig: results experiment 2} shows the solution of the differential equations for $\rhob^\circ(\cdot)$, summed over the vehicle types, with the error bars indicating the standard errors corresponding to each cell; the $i$-th panel corresponds with the situation after $300(i-1)$ seconds, for $i\in\{1,\ldots,6\}$. We conclude from the pictures that the approximation succeeds in replicating the phenomenon of the jam moving backward. 

It is also observed that the vehicle densities in the cells in front of the traffic jam have a particularly large variability. This is due to the fact that vehicles approach the jam relatively fast, resulting in a high in-flow rate for the cells in front of the jam.
Note that this behavior is consistent with a large scatter in the MFD at the onset of congestion. The cells behind the jam (with index between 12 and 20, that is) also have increased in-flow (and later out-flow), due to the dissipation of the traffic jam. As a consequence the variability (in terms of  the standard error) of the vehicle density increases, until the flow stabilizes, from which point on the standard error goes down again.
}\end{example}

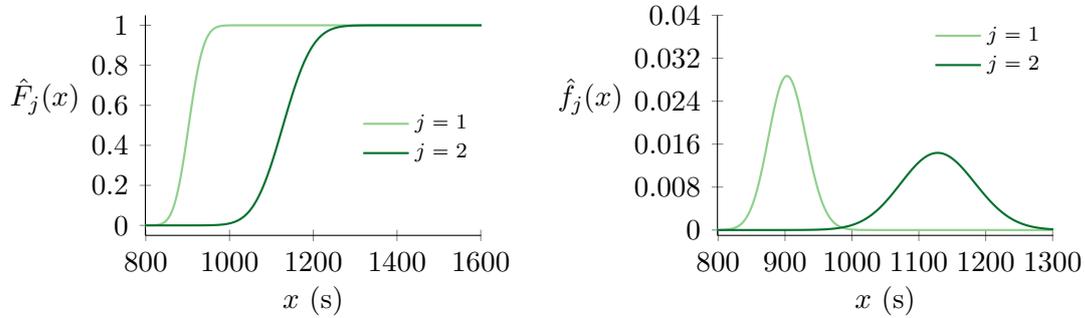
\begin{figure}
\begin{tabular}{cc}
\begin{tikzpicture}	
\begin{axis}[
    xlabel={$x$ (s)},    ylabel={$\hat{F}_{j}(x)$},
    xmin=799, xmax=1601,    ymin=-0.05, ymax=1.05,
    xtick={800,1000,1200,1400,1600},    ytick={0.0,0.2,0.4,0.6,0.8,1.0},
    xticklabel style={/pgf/number format/.cd,fixed,set thousands separator={}},
	scaled x ticks=false,
    xtick pos=left,    ytick pos=left,    axis lines=left,
    width = 6cm,    height = 4.5cm,
    no markers,    x axis line style=-,    y axis line style=-,
    y label style={at={(axis description cs:-0.02,.5)}, rotate = 270, anchor=south},
    font = \small,
    legend style={at={(1,0.6)},anchor=north east, draw=none, font=\tiny},
    legend cell align={left}
]
\addplot+[thick,color32] table [x={time}, y={F1}] {pd_TT_F_df.dat};
\addplot+[thick,color35] table [x={time}, y={F2}] {pd_TT_F_df.dat};
\legend{$j=1$,$j=2$}
\end{axis}
\end{tikzpicture} &
\begin{tikzpicture}	
\begin{axis}[
    xlabel={$x$ (s)},    ylabel={$\hat{f}_{j}(x)$},
    xmin=799, xmax=1301,    ymin=-0.001, ymax=0.04,
    xtick={800,900,1000,1100,1200,1300},    ytick={0.0,0.008,0.016,0.024,0.032,0.040},
    xticklabel style={/pgf/number format/.cd,fixed,set thousands separator={}},
    yticklabel style={/pgf/number format/fixed, /pgf/number format/precision=4},
	scaled x ticks=false,
	scaled y ticks=false,
    xtick pos=left,    ytick pos=left,    axis lines=left,
    width = 6cm,    height = 4.5cm,
    no markers,    x axis line style=-,    y axis line style=-,
    y label style={at={(axis description cs:-0.1,.5)}, rotate = 270, anchor=south},
    font = \small,
    legend style={at={(1,1)},anchor=north east, draw=none, font=\tiny},
    legend cell align={left}
]
\addplot+[thick,color32] table [x={time}, y={f1}] {pd_TT_ff_df.dat};
\addplot+[thick,color35] table [x={time}, y={f2}] {pd_TT_ff_df.dat};
\legend{$j=1$,$j=2$}
\end{axis}
\end{tikzpicture} 
\end{tabular}
\caption{Approximated cumulative distribution function (where $\hat F_j(x):={\mathbb P}(T_{10,49,j}(200)<x)$) and the corresponding density ($\hat{f}_{T_j}(x)$) of travel times for types $j = 1,2$, from cell~$10$ to cell~$50$, for cars leaving cell 10 at time $200$ s.}
\label{Fig: Travel Time Distribution and Density}
\end{figure}

\begin{example}[Shocks]\label{E4}{\em 

The  final experiment illustrates how our results replicate shock waves, an important topic in the theory of conservation laws \cite[Section 7]{leveque1992}. In this {case, }a fast small density reaches a larger and slower downstream density. 
In traffic flow theory, such shocks have been studied since \cite{rich1956}.

In our {experiment,} we take $d=40$ with $\ell_i = 0.6$~km for each $i$. We simulate for $500$~s with $\Delta t = 1$~s. Initially, cells 20 up to 24 have a density of $350$~veh/km, and cells 5 up to 9 have a density of $200$~veh/km, with every other cell being empty initially. 
\autoref{fig: results experiment 3} depicts the evolution of the densities. The small density catches up with the large density, and one can see how the densities merge at the back of the jamming density, while the front of the jamming density dissipates. }\end{example}

\begin{remark}{\em 
Importantly, it is now straightforward to apply  our results in a control context. For example, suppose that one wants to prevent congestion in cell $i$. By lowering the maximum speed in cells $1,\ldots,i-1$ we alter the shape of the MFD in these cells. With our {methods} we can assess to what extent this decreases the vehicle density in cell $i$. Moreover, the corresponding variance allows us to determine the maximum velocity to be imposed  in cells $1,\ldots,i-1$ 
to make sure   that the vehicle density of cell $i$ remains below some predefined acceptable level with probability of at most $\alpha\in(0,1)$.}
\end{remark}

\section{Extensions, Generalizations, Ramifications}
\label{EGR}

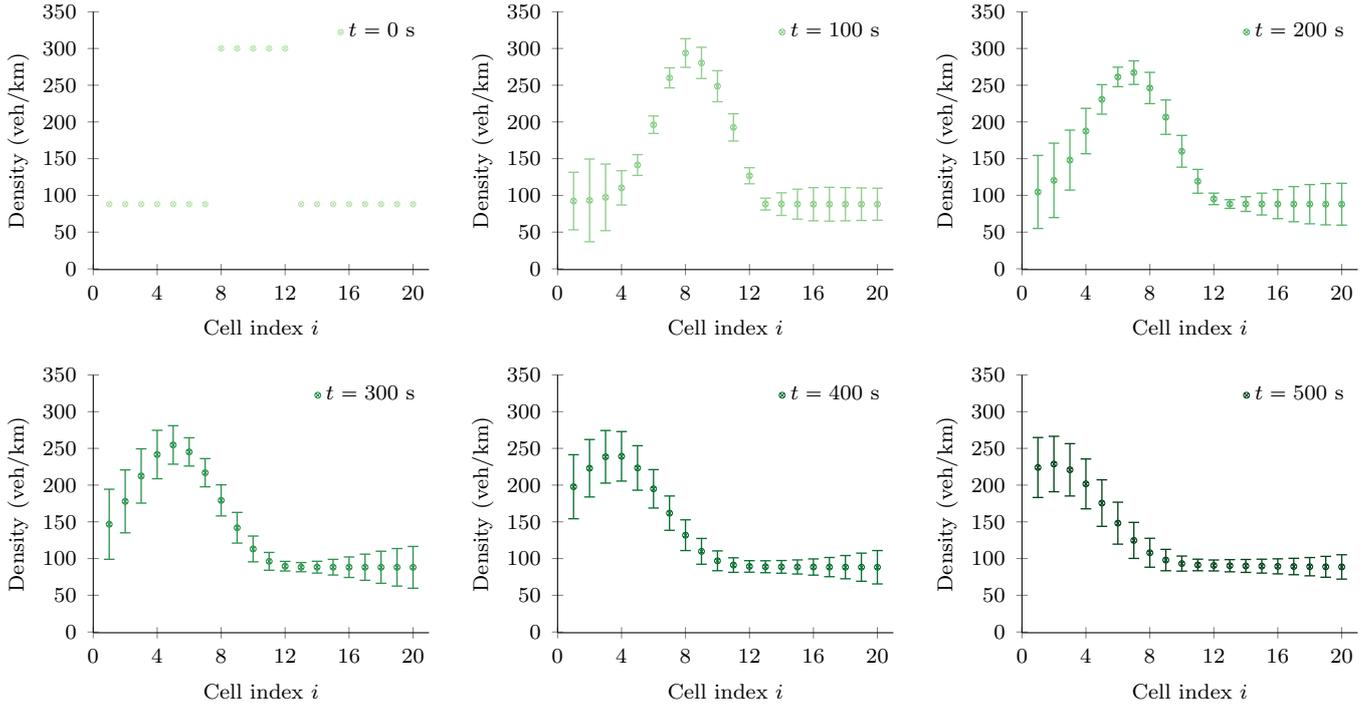
\begin{figure}
\begin{tabular}{ccc}
\hspace{-1.5cm}
\begin{tikzpicture}	
\begin{axis}[
    xlabel={Cell index $i$},    ylabel={Density (veh/km)},
    xmin=0, xmax=21,    ymin=-1, ymax=350,
    xtick={0,4,8,12,16,20},    ytick={0,50,100,150,200,250,300,350},
    xticklabel style={/pgf/number format/fixed, /pgf/number format/precision=2},
	scaled x ticks=false,
    xtick pos=left,    ytick pos=left,    axis lines=left,
    width = 6cm,    height = 5cm,
    mark size=1.0pt,
    only marks,    x axis line style=-,    y axis line style=-,
    y label style={at={(axis description cs:0.12,.5)},anchor=south},
    font = \tiny,
    legend style={at={(1,1)},anchor=north east, draw=none, font=\tiny},
    legend cell align={left}
]
\addplot+[color31,mark=otimes] table [x={space}, y={rho_t0}] {pd_CB2_sum_df.dat};
\legend{$t=0$ s}
\end{axis}
\end{tikzpicture}  &
\begin{tikzpicture}	
\begin{axis}[
    xlabel={Cell index $i$},    ylabel={Density (veh/km)},
    xmin=0, xmax=21,    ymin=-1, ymax=350,
    xtick={0,4,8,12,16,20},    ytick={0,50,100,150,200,250,300,350},
    xticklabel style={/pgf/number format/fixed, /pgf/number format/precision=2},
	scaled x ticks=false,
    xtick pos=left,    ytick pos=left,    axis lines=left,
    width = 6cm,    height = 5cm,
    mark size=1.1pt,
    only marks,    x axis line style=-,    y axis line style=-,
    y label style={at={(axis description cs:0.12,.5)},anchor=south},
    font = \tiny,
    legend style={at={(1,1)},anchor=north east, draw=none, font=\tiny},
    legend cell align={left}
]
\addplot+[color32,mark=otimes,error bars/.cd,y explicit,y dir=both,error bar style={line width=0.5pt},error mark options={
      rotate=90,
      mark size=2pt,
      line width=0.5pt
    }] table [x={space}, y={rho_t300}, y error = {rho_t300_error}] {pd_CB2_sum_df.dat};
\legend{$t=100$ s}
\end{axis}
\end{tikzpicture} &
\begin{tikzpicture}	
\begin{axis}[
    xlabel={Cell index $i$},    ylabel={Density (veh/km)},
    xmin=0, xmax=21,    ymin=-1, ymax=350,
    xtick={0,4,8,12,16,20},    ytick={0,50,100,150,200,250,300,350},
    xticklabel style={/pgf/number format/fixed, /pgf/number format/precision=2},
	scaled x ticks=false,
    xtick pos=left,    ytick pos=left,    axis lines=left,
    width = 6cm,    height = 5cm,
    mark size=1.1pt,
    only marks,    x axis line style=-,    y axis line style=-,
    y label style={at={(axis description cs:0.12,.5)},anchor=south},
    font = \tiny,
    legend style={at={(1,1)},anchor=north east, draw=none, font=\tiny},
    legend cell align={left}
]
\addplot+[color33,mark=otimes,error bars/.cd,y explicit,y dir=both,error bar style={line width=0.5pt},error mark options={
      rotate=90,
      mark size=2pt,
      line width=0.5pt
    }] table [x={space}, y={rho_t600}, y error = {rho_t600_error}] {pd_CB2_sum_df.dat};
\legend{$t=200$ s}
\end{axis}
\end{tikzpicture} \\
\hspace{-1.5cm} \begin{tikzpicture}	
\begin{axis}[
    xlabel={Cell index $i$},    ylabel={Density (veh/km)},
    xmin=0, xmax=21,    ymin=-1, ymax=350,
    xtick={0,4,8,12,16,20},    ytick={0,50,100,150,200,250,300,350},
    xticklabel style={/pgf/number format/fixed, /pgf/number format/precision=2},
	scaled x ticks=false,
    xtick pos=left,    ytick pos=left,    axis lines=left,
    width = 6cm,    height = 5cm,
    mark size=1.1pt,
    only marks,    x axis line style=-,    y axis line style=-,
    y label style={at={(axis description cs:0.12,.5)},anchor=south},
    font = \tiny,
    legend style={at={(1,1)},anchor=north east, draw=none, font=\tiny},
    legend cell align={left}
]
\addplot+[color34,mark=otimes,error bars/.cd,y explicit,y dir=both,error bar style={line width=0.5pt},error mark options={
      rotate=90,
      mark size=2pt,
      line width=0.5pt
    }] table [x={space}, y={rho_t900}, y error = {rho_t900_error}] {pd_CB2_sum_df.dat};
\legend{$t=300$ s}
\end{axis}
\end{tikzpicture} &
\begin{tikzpicture}	
\begin{axis}[
    xlabel={Cell index $i$},    ylabel={Density (veh/km)},
    xmin=0, xmax=21,    ymin=-1, ymax=350,
    xtick={0,4,8,12,16,20},    ytick={0,50,100,150,200,250,300,350},
    xticklabel style={/pgf/number format/fixed, /pgf/number format/precision=2},
	scaled x ticks=false,
    xtick pos=left,    ytick pos=left,    axis lines=left,
    width = 6cm,    height = 5cm,
    mark size=1.1pt,
    only marks,    x axis line style=-,    y axis line style=-,
    y label style={at={(axis description cs:0.12,.5)},anchor=south},
    font = \tiny,
    legend style={at={(1,1)},anchor=north east, draw=none, font=\tiny},
    legend cell align={left}
]
\addplot+[color35,mark=otimes,error bars/.cd,y explicit,y dir=both,error bar style={line width=0.5pt},error mark options={
      rotate=90,
      mark size=2pt,
      line width=0.5pt
    }] table [x={space}, y={rho_t1200}, y error = {rho_t1200_error}] {pd_CB2_sum_df.dat};
\legend{$t=400$ s}
\end{axis}
\end{tikzpicture} &
\begin{tikzpicture}	
\begin{axis}[
    xlabel={Cell index $i$},    ylabel={Density (veh/km)},
    xmin=0, xmax=21,    ymin=-1, ymax=350,
    xtick={0,4,8,12,16,20},    ytick={0,50,100,150,200,250,300,350},
    xticklabel style={/pgf/number format/fixed, /pgf/number format/precision=2},
	scaled x ticks=false,
    xtick pos=left,    ytick pos=left,    axis lines=left,
    width = 6cm,    height = 5cm,
    mark size=1.1pt,
    only marks,    x axis line style=-,    y axis line style=-,
    y label style={at={(axis description cs:0.12,.5)},anchor=south},
    font = \tiny,
    legend style={at={(1,1)},anchor=north east, draw=none, font=\tiny},
    legend cell align={left}
]
\addplot+[color36,mark=otimes,error bars/.cd,y explicit,y dir=both,error bar style={line width=0.5pt},error mark options={
      rotate=90,
      mark size=2pt,
      line width=0.5pt
    }] table [x={space}, y={rho_t1500}, y error = {rho_t1500_error}] {pd_CB2_sum_df.dat};
\legend{$t=500$ s}
\end{axis}
\end{tikzpicture}
\end{tabular}
\caption{Solution differential equations for $\rhob(t)$ and $\textrm{var}[\rhoh(t)]$, Example \ref{E3} - backward moving jam.}
\label{fig: results experiment 2}
\end{figure}

In this {section,} we discuss a number of possible and relevant extensions, generalizations and ramifications. To make the exposition as clear as possible, till this point we have worked with the basic model of a single road segment and a simple form for the infinitesimal generator of the underlying Markov chain. As we will make clear below, however, a considerably broader class of systems can be dealt with analogously. 

\subsection{Networks}
In the setup presented, we restricted ourselves to a single segment consisting of cells, each of them endowed with its own discrete flux-function. This flux-function, say from cell $i$ to cell $i+1$, was assumed to depend on the vehicle densities in the sending cell $i$ as well as the destination cell $i+1$. Upon inspecting the framework in \cite[Ch.\ 8]{kurtz1981approximation}, however, one sees that it is allowed to consider flux-functions that depend on the vehicle densities in \textit{all} cells in the segment. This observation makes it straightforward to generalize our framework to road segments, each consisting of a number of cells, that are embedded in a network (with nodes that connect  the first and last cells of the  segments). By extending \autoref{Assumption: MFD}, so that each $\tilde q_i$  is a function of the densities in an arbitrary set of cells, the fluid limit and diffusion limit results carry over, as well as the results on the travel times. 

As a special case of the network structure discussed above, on-ramps and off-ramps (to and from our road segment) can be covered. Considering a segment consisting of $d$ cells, 
if a ramp has to be placed adjacent to cell $i$, then this can be done 
by adding an extra cell, outside the segment, next to the cell~$i$. Depending on the direction of flow in the new cell, traffic will either merge from cell~$i-1$ and the new cell into cell~$i$, or it will diverge from cell~$i-1$ into cell~$i$ and the new cell, thus encoding whether it concerns an on-ramp or an off-ramp. 

Of course, it is not immediately clear that, for a general network and an arbitrary MFD,  there exists a unique solution of the corresponding Riemann problem (now on a network), that in addition satisfies the regularity properties of  \autoref{Assumption: MFD}. In \cite[Section 2.3]{DaganzoNetworks}, it is explained how an arbitrary network  can be captured by  a network that only connects two upstream cells to one downstream cell, or vice versa; this reduction property  simplifies the problem of finding a discrete flux-function considerably. In particular, it features traffic merging from two cells into one, or traffic diverging from one into two cells (see \cite{DaganzoNetworks} for a schematic representation), making it particularly suitable for modelling on-ramps and off-ramps on a segment. In the remainder of \cite{DaganzoNetworks}, the discrete flux-function for the Daganzo MFD (i.e., our first example in \autoref{sec: example MFDs}) is found. Moreover, under this simplification of the network structure, one can solve the single-class case under fairly general conditions on the MFD; cf.\ \cite[Ch.\ 4, 5]{garavello2006}. For the multi-class case, to the best of the authors' knowledge, no general results {exist}, but there are  solutions for specific MFDs  (see \cite{ngoduy2010} and references therein). Likewise, for other proposed multi-class MFDs, which are often functions with strong regularity properties, one would envisage that solutions to the network Riemann problem can be found.

The generality of our model, and the flexibility of the framework of \cite[Ch.\ 8]{kurtz1981approximation}, {make} it possible to include behavior of vehicles that is specific to merging and diverging situations. More specifically, since discrete flux-functions in our model can be cell-dependent, we can model merging and diverging behavior of vehicles (or even more complicated flow structures), using special `merge' and `diverge' flux-functions. As long as these functions satisfy the regularity conditions in \autoref{Assumption: MFD}, the results from Sections \ref{sec: scaling limits} and \ref{sec: Travel Times} go through.

Finally, we can also explicitly model the routing of the vehicles through the network. In traffic flow theory, there are two major ways of  routing vehicles on networks. The first is to route deterministically. This can be done by working with origin-destination  pairs, where every vehicle has a fixed route from the  origin to the  destination; see{, e.g.,~}\cite[Ch.\ 7]{garavello2006}. In our setup we can incorporate this mechanism by treating vehicles using a specific route as a type, with suitable boundary conditions (describing the arrivals at the origin and departures at the destination). The second approach  is to route probabilistically,  see{, e.g.,~}\cite{DaganzoNetworks} or \cite[Section 5]{garavello2006}; think of the turning rates of vehicles at junctions. In the case where a number of flows merges into one flow at a junction, and when the system is saturated at the point of merging, a probability distribution is used to describe the contribution of the merged flows to the total flow out of the junction. This probability distribution  controls the rate at which vehicles from each of the merging cells yield to vehicles from the other cells.


%

\begin{figure}
\begin{tabular}{ccc}
\hspace{-1.5cm}
\begin{tikzpicture}	
\begin{axis}[
    xlabel={Cell index $i$},    ylabel={Density (veh/km)},
    xmin=0, xmax=41,    ymin=-10, ymax=370,
    xtick={0,5,10,15,20,25,30,35,40},    ytick={0,50,100,150,200,250,300,350},
    xticklabel style={/pgf/number format/fixed, /pgf/number format/precision=2},
	scaled x ticks=false,
    xtick pos=left,    ytick pos=left,    axis lines=left,
    width = 6cm,    height = 5cm,
    mark size=1.0pt,
    only marks,    x axis line style=-,    y axis line style=-,
    y label style={at={(axis description cs:0.12,.5)},anchor=south},
    font = \tiny,
    legend style={at={(1,1)},anchor=north east, draw=none, font=\tiny},
    legend cell align={left}
]
\addplot+[color31,mark=otimes] table [x={space}, y={rho_t0}] {pd_CB3_sum_df.dat};
\legend{$t=0$ s}
\end{axis}
\end{tikzpicture} &
\begin{tikzpicture}	
\begin{axis}[
    xlabel={Cell index $i$},    ylabel={Density (veh/km)},
    xmin=0, xmax=41,    ymin=-10, ymax=370,
    xtick={0,5,10,15,20,25,30,35,40},    ytick={0,50,100,150,200,250,300,350},
    xticklabel style={/pgf/number format/fixed, /pgf/number format/precision=2},
	scaled x ticks=false,
    xtick pos=left,    ytick pos=left,    axis lines=left,
    width = 6cm,    height = 5cm,
    mark size=1.1pt,
    only marks,    x axis line style=-,    y axis line style=-,
    y label style={at={(axis description cs:0.12,.5)},anchor=south},
    font = \tiny,
    legend style={at={(1,1)},anchor=north east, draw=none, font=\tiny},
    legend cell align={left}
]
\addplot+[color32,mark=otimes,error bars/.cd,y explicit,y dir=both,error bar style={line width=0.5pt},error mark options={
      rotate=90,
      mark size=2pt,
      line width=0.5pt
    }] table [x={space}, y={rho_t100}, y error = {rho_t100_error}] {pd_CB3_sum_df.dat};
\legend{$t=100$ s}
\end{axis}
\end{tikzpicture} &
\begin{tikzpicture}	
\begin{axis}[
    xlabel={Cell index $i$},    ylabel={Density (veh/km)},
    xmin=0, xmax=41,    ymin=-10, ymax=370,
    xtick={0,5,10,15,20,25,30,35,40},    ytick={0,50,100,150,200,250,300,350},
    xticklabel style={/pgf/number format/fixed, /pgf/number format/precision=2},
	scaled x ticks=false,
    xtick pos=left,    ytick pos=left,    axis lines=left,
    width = 6cm,    height = 5cm,
    mark size=1.1pt,
    only marks,    x axis line style=-,    y axis line style=-,
    y label style={at={(axis description cs:0.12,.5)},anchor=south},
    font = \tiny,
    legend style={at={(1,1)},anchor=north east, draw=none, font=\tiny},
    legend cell align={left}
]
\addplot+[color33,mark=otimes,error bars/.cd,y explicit,y dir=both,error bar style={line width=0.5pt},error mark options={
      rotate=90,
      mark size=2pt,
      line width=0.5pt
    }] table [x={space}, y={rho_t200}, y error = {rho_t200_error}] {pd_CB3_sum_df.dat};
\legend{$t=200$ s}
\end{axis}
\end{tikzpicture}\\
\hspace{-1.5cm}
\begin{tikzpicture}	
\begin{axis}[
    xlabel={Cell index $i$},    ylabel={Density (veh/km)},
    xmin=0, xmax=41,    ymin=-10, ymax=370,
    xtick={0,5,10,15,20,25,30,35,40},    ytick={0,50,100,150,200,250,300,350},
    xticklabel style={/pgf/number format/fixed, /pgf/number format/precision=2},
	scaled x ticks=false,
    xtick pos=left,    ytick pos=left,    axis lines=left,
    width = 6cm,    height = 5cm,
    mark size=1.1pt,
    only marks,    x axis line style=-,    y axis line style=-,
    y label style={at={(axis description cs:0.12,.5)},anchor=south},
    font = \tiny,
    legend style={at={(1,1)},anchor=north east, draw=none, font=\tiny},
    legend cell align={left}
]
\addplot+[color34,mark=otimes,error bars/.cd,y explicit,y dir=both,error bar style={line width=0.5pt},error mark options={
      rotate=90,
      mark size=2pt,
      line width=0.5pt
    }] table [x={space}, y={rho_t300}, y error = {rho_t300_error}] {pd_CB3_sum_df.dat};
\legend{$t=300$ s}
\end{axis}
\end{tikzpicture} &
\begin{tikzpicture}	
\begin{axis}[
    xlabel={Cell index $i$},    ylabel={Density (veh/km)},
    xmin=0, xmax=41,    ymin=-10, ymax=370,
    xtick={0,5,10,15,20,25,30,35,40},    ytick={0,50,100,150,200,250,300,350},
    xticklabel style={/pgf/number format/fixed, /pgf/number format/precision=2},
	scaled x ticks=false,
    xtick pos=left,    ytick pos=left,    axis lines=left,
    width = 6cm,    height = 5cm,
    mark size=1.1pt,
    only marks,    x axis line style=-,    y axis line style=-,
    y label style={at={(axis description cs:0.12,.5)},anchor=south},
    font = \tiny,
    legend style={at={(1,1)},anchor=north east, draw=none, font=\tiny},
    legend cell align={left}
]
\addplot+[color35,mark=otimes,error bars/.cd,y explicit,y dir=both,error bar style={line width=0.5pt},error mark options={
      rotate=90,
      mark size=2pt,
      line width=0.5pt
    }] table [x={space}, y={rho_t400}, y error = {rho_t400_error}] {pd_CB3_sum_df.dat};
\legend{$t=400$ s}
\end{axis}
\end{tikzpicture} & 
\begin{tikzpicture}	
\begin{axis}[
    xlabel={Cell index $i$},    ylabel={Density (veh/km)},
    xmin=0, xmax=41,    ymin=-10, ymax=370,
    xtick={0,5,10,15,20,25,30,35,40},    ytick={0,50,100,150,200,250,300,350},
    xticklabel style={/pgf/number format/fixed, /pgf/number format/precision=2},
	scaled x ticks=false,
    xtick pos=left,    ytick pos=left,    axis lines=left,
    width = 6cm,    height = 5cm,
    mark size=1.1pt,
    only marks,    x axis line style=-,    y axis line style=-,
    y label style={at={(axis description cs:0.12,.5)},anchor=south},
    font = \tiny,
    legend style={at={(1,1)},anchor=north east, draw=none, font=\tiny},
    legend cell align={left}
]
\addplot+[color36,mark=otimes,error bars/.cd,y explicit,y dir=both,error bar style={line width=0.5pt},error mark options={
      rotate=90,
      mark size=2pt,
      line width=0.5pt
    }] table [x={space}, y={rho_t500}, y error = {rho_t500_error}] {pd_CB3_sum_df.dat};
\legend{$t=500$ s}
\end{axis}
\end{tikzpicture}
\end{tabular}
\caption{Solution differential equations for $\rhob(t)$ and $\textrm{var}[\rhoh(t)]$, Example \ref{E4} - Shocks.}
\label{fig: results experiment 3}
\end{figure}
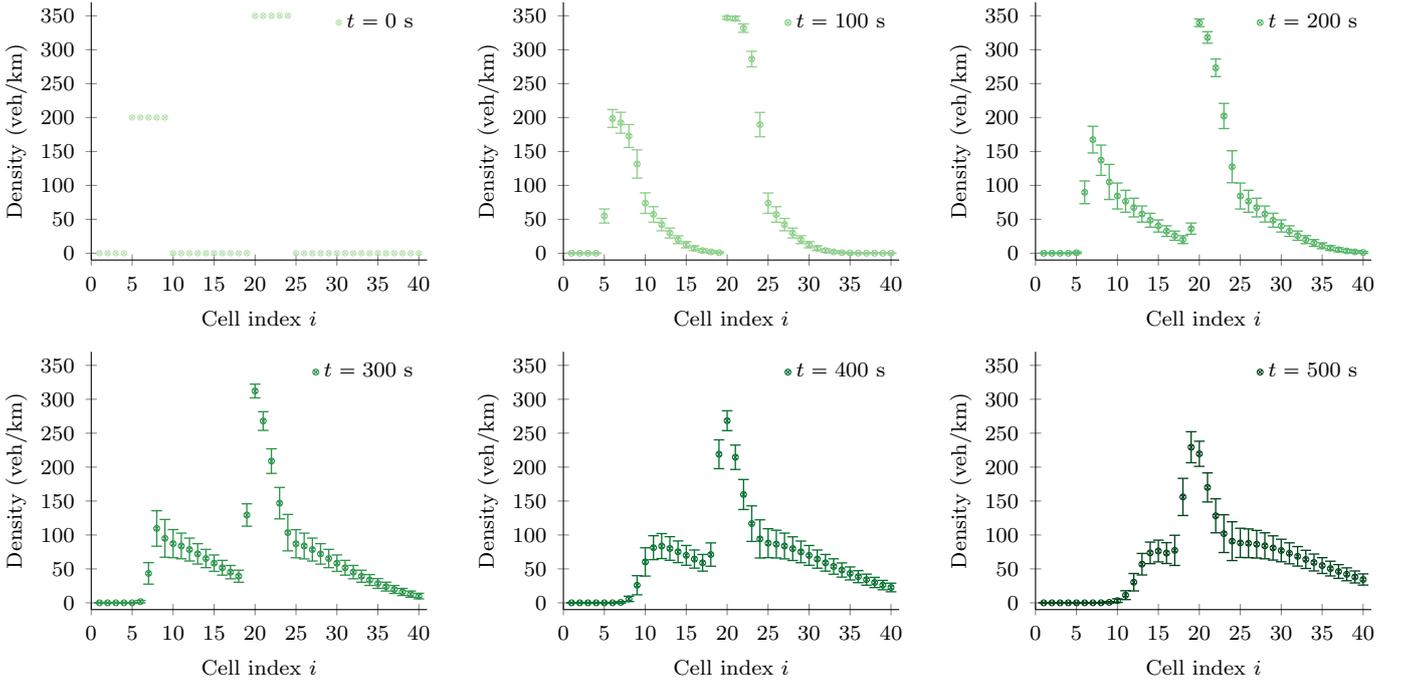

\subsection{Generalized transition rates}

In practice, the model's stochastic dynamics need not be homogeneous in time. One would{, for instance,} wish to accommodate a daily pattern with rush hours and more quiet periods. This can be realized by making the arrival rate (to cell 1) and departure rates (from cell $d$) 
time-dependent. To preserve the results in Sections~\ref{sec: scaling limits} and~\ref{sec: Travel Times}, we can scale the rates of  $Y(\cdot)$ by a factor $n$ (instead of scaling time), besides the scaling of the cell lengths; in the setup we have studied in the previous sections, in which the Poisson processes do not depend on time, scaling time and scaling the rates of  $Y(\cdot)$ are equivalent. To preserve the time-dependent arrival pattern though, we let the scaled process counting the number of type-$j$ arrivals to cell~$1$ be given by
\[
	\frac{1}{n \ell_i} Y_{0,j}\left( n \int_0^{t} \lambda_j(s) \wedge \sup_{x \in\mathbb{R}_+^k} \tilde q_{0,j}\left(x,(\rho^n_{1k}(s))_k \right) \diff s \right), \quad j \in \{1,\ldots,m\},
\]
where $\lambda_j(s)$ is the intensity of the type-$j$ Poisson arrival process at time $s \geq 0$. Clearly, $\lambda_j(s)$ does not depend on $\rho^n(t)$, so that the minimum of these functions is still Lipschitz in $\rho^n(t)$ by \autoref{Assumption: MFD}.

Additionally, we can let the discrete flux-function from \autoref{Assumption: MFD} be stochastic itself, as was suggested in \cite{QU2017}. As long as the fluid limit of the associated Poisson processes is  deterministic and continuous, the results from Sections~\ref{sec: scaling limits} and \ref{sec: Travel Times} remain valid, albeit with a different limit. One typically gets that in the fluid limit all quantities are replaced by their time-average counterpart, whereas in the 
diffusion limit one obtains a larger variance, as a consequence of  the uncertainty that has been added to the model. One possible way to obtain a stochastic MFD would be to periodically resample the MFD's parameters.
Another option is to let the model's transition rates depend on exogenous influences (for instance the weather). This could be achieved by introducing an external (e.g.\ Markovian) background process, where the state of this process determines the current value of the transition rates. We refer to{, e.g.,}  \cite{Marijn} for a derivation of fluid and diffusion limits in a related type of networks, viz.\ a network of Markov-modulated infinite-server queues. In other {contexts,} a similar methodology has been used{, in e.g.,} \cite{gang,spreij2019}.

Finally, we can relax of the assumption of the cell-transition times being 
exponentially distributed to being phase-type distributed. This is a useful extension, as any non-negative distribution can be approximated arbitrarily closely by an appropriately chosen phase-type distribution. It means that the time until a vehicle jumps to the next cell consists of multiple  phases, where the time spent in each of the phases should be exponentially distributed. The transition rates should depend on the number of vehicles in each phase, both in the sending cell and the receiving cell, aggregated over the phases (so as to reflect the per-cell vehicle density). Importantly, such an extension still fits the framework developed in \cite{kurtz1981approximation}, so that the results from \autoref{sec: scaling limits} and \autoref{sec: Travel Times} carry over.

\section{Concluding Remarks}\label{CR}
In this {paper,} we have developed a Markovian model for vehicle densities in a multiple-type road traffic network, with dynamics that are consistent with state-of-the-art traffic-flow models. As this model does not allow any explicit analysis, we have resorted to an asymptotic framework. More concretely, we have established a fluid and diffusion limit: scaling the lengths of the cells and time by $n$, and appropriately centering and normalizing the vehicle densities, the resulting process converges, as $n\to\infty$, to a Gaussian process. This diffusion limit can be used to produce an approximation for the vehicle density distribution in time. Along the same lines, an approximation for the travel-time distribution has been developed. In a set of numerical experiments we have concluded that  the resulting approximations are highly accurate; in addition, we have shown that our model is capable of reproducing various known traffic phenomena.

{\subsubsection*{Acknowledgments}
The authors are grateful to Dieter Fiems (University of Ghent) for helpful suggestions regarding the  travel-time analysis. We also would like to thank Sandjai Bhulai (Vrije Universiteit Amsterdam) for providing useful feedback.}

\bibliographystyle{abbrv} 
\bibliography{segmentbibliography}

\begin{thebibliography}{10}

\bibitem{BC}
S.~Benzoni-Gavage and R.~Colombo.
\newblock An $n$-populations model for traffic flow.
\newblock {\em European Journal of Applied Mathematics}, 14:587--612, 2003.

\bibitem{bremaud1981}
P.~Br{\'e}maud.
\newblock {\em Point Processes and Queues}.
\newblock Springer Series in Statistics. Springer New York, 1981.

\bibitem{cb2003}
S.~Chanut and C.~Buisson.
\newblock Macroscopic model and its numerical solution for two-flow mixed
  traffic with different speeds and lengths.
\newblock {\em Transportation Research Record}, 1852:209--219, 2003.

\bibitem{daganzo1994cell}
C.~Daganzo.
\newblock The cell transmission model: A dynamic representation of highway
  traffic consistent with the hydrodynamic theory.
\newblock {\em Transportation Research, Part B: Methodological}, 28:269--287,
  1994.

\bibitem{DaganzoNetworks}
C.~Daganzo.
\newblock The cell transmission model, part {II}: Network traffic.
\newblock {\em Transportation Research Part B: Methodological}, 29:79 -- 93,
  1995.

\bibitem{DA95}
C.~Daganzo.
\newblock Requiem for second-order fluid approximations of traffic flow.
\newblock {\em Transportation Research, Part B: Methodological}, 29:277--286,
  1995.

\bibitem{drake}
J.~Drake, J.~Schofer, and A.~May.
\newblock A statistical analysis of speed-density hypotheses.
\newblock {\em Highway Research Record}, 154:53--87, 1967.

\bibitem{evans10}
L.~Evans.
\newblock {\em Partial Differential Equations}.
\newblock American Mathematical Society, Providence, R.I., 2010.

\bibitem{garavello2006}
M.~Garavello and B.~Piccoli.
\newblock {\em Traffic Flow on Networks}.
\newblock 01 2006.

\bibitem{green}
B.~Greenshields.
\newblock The photographic method of studying traffic behavior.
\newblock In {\em Proceedings of the 13th Annual Meeting of the Highway
  Research Board}, 1934.

\bibitem{hale1969}
J.~Hale.
\newblock {\em Ordinary Differential Equations}.
\newblock Pure and applied mathematics : a series of texts and monographs.
  Wiley-Interscience, 1969.

\bibitem{gang}
G.~Huang, H.~Jansen, M.~Mandjes, P.~Spreij, and K.~{De Turck}.
\newblock {M}arkov-modulated {O}rnstein-{U}hlenbeck processes.
\newblock {\em Advances in Applied Probability}, 48:235--254, 2016.

\bibitem{JL2012}
S.~Jabari and H.~Liu.
\newblock A stochastic model of traffic flow: Theoretical foundations.
\newblock {\em Transportation Research, Part B: Methodological}, 46:156--174,
  2012.

\bibitem{JL2013}
S.~Jabari and H.~Liu.
\newblock A stochastic model of traffic flow: Gaussian approximation and
  estimation.
\newblock {\em Transportation Research, Part B: Methodological}, 47:15--41,
  2013.

\bibitem{js2003}
J.~Jacod and A.~Shiryaev.
\newblock {\em Limit Theorems for Stochastic Processes}, volume 288.
\newblock Springer Science \& Business Media, 2003.

\bibitem{Marijn}
H.~Jansen, M.~Mandjes, K.~de~Turck, and S.~Wittevrongel.
\newblock Diffusion limits for networks of {M}arkov-modulated infinite-server
  queues.
\newblock {\em Performance Evaluation}, 135:102039, 2019.

\bibitem{karatzas2012brownian}
I.~Karatzas and S.~Shreve.
\newblock {\em Brownian Motion and Stochastic Calculus}, volume 113.
\newblock Springer Science \& Business Media, 2012.

\bibitem{kern}
B.~Kerner.
\newblock {\em The Physics of Traffic: Empirical Freeway Pattern Features,
  Engineering Applications, and Theory}.
\newblock Springer, 2004.

\bibitem{kurtz1981approximation}
T.~Kurtz.
\newblock {\em Approximation of {P}opulation {P}rocesses}, volume~36.
\newblock SIAM, 1981.

\bibitem{leveque1992}
R.~LeVeque.
\newblock {\em Numerical Methods for Conservation Laws}, volume 132.
\newblock Springer, 1992.

\bibitem{lw1955}
M.~Lighthill and G.~Whitham.
\newblock On kinematic waves. {I}: {F}lood movement in long rivers. {II}: {A}
  theory of traffic flow on long crowded roads.
\newblock {\em Proceedings of the Royal Society}, 229A:281--345, 1955.

\bibitem{li2008}
S.~Logghe and L.~Immers.
\newblock Multi-class kinematic wave theory of traffic flow.
\newblock {\em Transportation Research, Part B: Methodological}, 42:523--541,
  2008.

\bibitem{maerivoet2005traffic}
S.~Maerivoet and B.~De~Moor.
\newblock Traffic flow theory.
\newblock {\em arXiv preprint physics/0507126}, 2005.

\bibitem{A1}
A.~Mandelbaum, W.~Massey, and M.~Reiman.
\newblock Strong approximations for {M}arkovian service networks.
\newblock {\em Queueing Systems}, 30:149--201, 1998.

\bibitem{A2}
W.~Massey and J.~Pender.
\newblock Gaussian skewness approximation for dynamic rate multi-server queues
  with abandonment.
\newblock {\em Queueing Systems}, 75:243--277, 2013.

\bibitem{ngoduy2010}
D.~Ngoduy.
\newblock Multiclass first-order modelling of traffic networks using
  discontinuous flow-density relationships.
\newblock {\em Transportmetrica}, 6:121--141, 2010.

\bibitem{NGO}
D.~Ngoduy and R.~Liu.
\newblock Multiclass first-order simulation model to explain non-linear traffic
  phenomena.
\newblock {\em Physica A: Statistical Mechanics and its Applications},
  385:667--682, 2007.

\bibitem{Qian2017}
Z.~Qian, J.~Li, X.~Li, M.~Zhang, and H.~Wang.
\newblock Modeling heterogeneous traffic flow: A pragmatic approach.
\newblock {\em Transportation Research Part B: Methodological}, 99:183--204, 05
  2017.

\bibitem{QU2017}
X.~Qu, J.~Zhang, and S.~Wang.
\newblock On the stochastic fundamental diagram for freeway traffic: Model
  development, analytical properties, validation, and extensive applications.
\newblock {\em Transportation Research, Part B: Methodological}, 104:256 --
  271, 2017.

\bibitem{rich1956}
P.~Richards.
\newblock Shock waves on the highway.
\newblock {\em Operations Research}, 4:42--51, 1956.

\bibitem{smul}
S.~Smulders.
\newblock Control of freeway traffic flow by variable speed signs.
\newblock {\em Transportation Research, Part B: Methodological}, 24:111--132,
  1990.

\bibitem{spreij2019}
P.~Spreij and P.~J. Storm.
\newblock Diffusion limits for a {M}arkov modulated binomial counting process.
\newblock {\em Probability in the Engineering and Informational Sciences},
  2019.

\bibitem{WageningenKessels2015genealogy}
F.~van Wageningen-Kessels, H.~Van~Lint, K.~Vuik, and S.~Hoogendoorn.
\newblock Genealogy of traffic flow models.
\newblock {\em EURO Journal on Transportation and Logistics}, 4:445--473, 2015.

\bibitem{ww2002}
G.~Wong and S.~Wong.
\newblock A multi-class traffic flow model--an extension of lwr model with
  heterogeneous drivers.
\newblock {\em Transportation Research, Part A: Policy and Practice},
  36:827--841, 2002.

\bibitem{ZHA}
P.~Zhang, R.-X. Liu, S.Wong, and S.-Q. Dai.
\newblock Hyperbolicity and kinematic waves of a class of multi-population
  partial differential equations.
\newblock {\em European Journal of Applied Mathematics}, 17:171--200, 2006.

\end{thebibliography}

\end{document}